\documentclass[11pt,twoside,reqno]{amsart}

\usepackage{microtype}
\usepackage[OT1]{fontenc}
\usepackage{type1cm}
\usepackage{amssymb}
\usepackage{enumerate}
\usepackage{xcolor}
\usepackage{mathrsfs}
\usepackage{dsfont}
\usepackage{multicol}
\usepackage{tikz}
\usetikzlibrary{calc,patterns,angles,quotes,arrows,decorations,positioning,decorations.pathmorphing,decorations.text}
\usepackage{pgfplots}
\pgfplotsset{compat=newest}
\usepackage{verbatim}

\usepackage{cite}

\usepackage[left=2.3cm,top=3cm,right=2.3cm]{geometry}
\numberwithin{equation}{section}

\theoremstyle{plain}
\newtheorem{theorem}{Theorem}[section]

\newtheorem{proposition}[theorem]{Proposition}

\newtheorem{lemma}[theorem]{Lemma}

\theoremstyle{remark}

\newtheorem{example}[theorem]{Example}

\theoremstyle{definition}

\newcommand{\MM}{\mathcal{M}}

\newcommand{\CC}{\mathcal{C}}

\newcommand{\R}{\mathbb{R}}
\newcommand{\RP}{\mathbb{RP}^1}

\newcommand{\N}{\mathbb{N}}
\newcommand{\hhh}{\mathtt{h}}
\newcommand{\iii}{\mathtt{i}}
\newcommand{\jjj}{\mathtt{j}}
\newcommand{\kkk}{\mathtt{k}}

\newcommand{\eps}{\varepsilon}
\newcommand{\fii}{\varphi}
\newcommand{\roo}{\varrho}

\newcommand{\A}{\mathsf{A}}

\newcommand{\dd}{\,\mathrm{d}}

\renewcommand{\ge}{\geqslant}
\renewcommand{\le}{\leqslant}
\renewcommand{\geq}{\geqslant}
\renewcommand{\leq}{\leqslant}

\DeclareMathOperator{\dimh}{dim_H}

\DeclareMathOperator{\diml}{dim_L}

\DeclareMathOperator{\dima}{dim_A}

\DeclareMathOperator{\cdima}{\mathcal{C}dim_A}
\DeclareMathOperator{\cdimh}{\mathcal{C}dim_H}

\DeclareMathOperator{\Ker}{Ker}
\DeclareMathOperator{\Tan}{Tan}
\DeclareMathOperator{\linspan}{span}
\DeclareMathOperator{\dist}{dist}
\DeclareMathOperator{\diam}{diam}
\DeclareMathOperator{\proj}{proj}
\DeclareMathOperator{\nproj}{\overline{proj}}
\DeclareMathOperator{\conv}{conv}

\DeclareMathOperator{\spt}{spt}
\DeclareMathOperator{\rank}{rank}

\renewcommand{\atop}[2]{\genfrac{}{}{0pt}{}{#1}{#2}}

\begin{document}

\title{Assouad dimension of planar self-affine sets}

\author{Bal\'azs B\'ar\'any}
\address[Bal\'azs B\'ar\'any]
        {Budapest University of Technology and Economics \\
         Department of Stochastics \\
         MTA-BME Stochastics Research Group \\
         P.O.\ Box 91 \\
         1521 Budapest \\
         Hungary}
\email{balubsheep@gmail.com}

\author{Antti K\"aenm\"aki}
\address[Antti K\"aenm\"aki]
        {Department of Physics and Mathematics \\
         University of Eastern Finland \\
         P.O.\ Box 111 \\
         FI-80101 Joensuu \\
         Finland}
\email{antti.kaenmaki@uef.fi}

\author{Eino Rossi}
\address[Eino Rossi]
        {Department of Mathematics and Statistics \\
        P.O.\ Box 68  (Pietari Kalmin katu 5) \\
        FI-00014 University of Helsinki\\
        Finland}
\email{eino.rossi@gmail.com}

\date{\today}
\thanks{B\'ar\'any acknowledges support from the grants OTKA K123782, NKFI PD123970, the J\'anos Bolyai Research Scholarship of the Hungarian Academy of Sciences, and the New National Excellence Program of the Ministry of Human Capacities \'UNKP-18-4-BME-385. K\"aenm\"aki thanks the Academy of Finland (project no.\ 286877) for financial support. Rossi was funded by the University of Helsinki via the project ``Quantitative rectifiability of sets and measures in Euclidean Spaces and Heisenberg groups'' (project no.\ 7516125).}

\begin{abstract}
  We calculate the Assouad dimension of a planar self-affine set $X$ satisfying the strong separation condition and the projection condition and show that $X$ is minimal for the conformal Assouad dimension. Furthermore, we see that such a self-affine set $X$ adheres to very strong tangential regularity by showing that any two points of $X$, which are generic with respect to a self-affine measure having simple Lyapunov spectrum, share the same collection of tangent sets.
\end{abstract}

\maketitle

\section{Introduction}

The goal of the paper is to calculate the Assouad dimension of planar self-affine sets satisfying the strong separation condition and the projection condition. Roughly speaking, the assumptions require that the self-affine set is constructed by using disjoint construction pieces such that it projects to a line segment for sufficiently many directions. While traditionally the Assouad dimension has been used to study quasiconformal mappings and embeddability problems, it has recently gained a lot of interest in fractal geometry; see e.g.\ \cite{Fraser2014,FraserHendersonOlsonRobinson2015,FraserHowroyd2017,FraserJordan2017,KOR, Mackay}. The Assouad dimension of a set is the maximal dimension possible to obtain by looking at coverings. It serves as an upper bound for the Hausdorff dimension.

Dimension theory on self-affine sets is an active research topic and during recent years, it has progressed a lot; see e.g.\ \cite{barany2017hausdorff,barany2018birkhoff,BaranyKaenmaki2017,BaranyKaenmakiKoivusalo2017,Feng2019,FraserHowroyd2017,FraserJordan2017,KOR,kolossvary2018triangular,Mackay,MorrisShmerkin2019}. It is currently not known how the Assouad dimension of a self-affine set and the affinity dimension, a natural upper bound for all the other standard dimensions, are related. Heuristic arguments suggest that in general, the Assouad dimension is strictly larger than the affinity dimension. Therefore, it is not possible to apply the methods which are usually used to study dimensions on self-affine sets. To the best of our knowledge, the works \cite{FraserHowroyd2017,FraserJordan2017,Mackay} are the sole papers addressing this question. They all consider the problem on different types of self-affine carpets: the standard carpets and sponges are studied in \cite{Mackay} and \cite{FraserHowroyd2017}, respectively, whereas the setting in \cite{FraserJordan2017} allows more freedom in the placement of the construction pieces while at the same time, requires the pieces to have the same shape, i.e.\ assumes homogeneity.

It was recently proved in \cite{KOR} that the Assouad dimension of a compact set can equivalently be defined to be the maximal Hausdorff dimension of weak tangent sets, Hausdorff limits of successive magnifications of the set. This introduces a method to address the problem we are considering. Indeed, we will develop a machinery to study the tangential structure of self-affine sets. To give some intuition, the reader familiar with the Ledrappier-Young theory for measures on self-affine sets (see \cite{Barany2015,BaranyKaenmaki2017,Feng2019}) may interpret this machinery as a Ledrappier-Young theory for self-affine sets. The Ledrappier-Young theory guarantees that the dimension of a measure is the sum of the dimensions of the projection and a generic slice, whereas in our case, we similarly conclude that the Assouad dimension is the sum of the dimensions of the projection and the maximal slice.

As a first concrete outcome of our considerations, we show that generic points of a self-affine set share the same collection of tangent sets. While the observation improves the results in \cite{BandtKaenmaki2013,KaenmakiKoivusaloRossi2017}, it also reveals that self-affine sets adhere to very strong tangential regularity. Furthermore, by relying on the developed machinery, under the strong separation condition and the projection condition, we manage to calculate the Assouad dimension for a large class of self-affine sets which include self-affine sets defined by dominated and strongly irreducible matrices, and simultaneously non-diagonalizable upper-triangular matrices having first diagonal element strictly larger than the second one. Our theorem thus is a notable generalization of the earlier results on this topic. Finally, by finding a tangent set with almost maximal Hausdorff dimension, we generalize the results in \cite{KOR,Mackay} by showing that the self-affine sets considered in this paper are minimal for the conformal Assouad dimension.

We refer the impatient reader to Section \ref{sec:main-results} where we have collected the main results. Section \ref{sec:preli} is devoted to preliminaries and the proofs of the main results can be found in Sections \ref{sec:tangent-sets} and \ref{sec:assouad-dimension}.

\section{Preliminaries} \label{sec:preli}

\subsection{Shift space}
Let $\Sigma = \{ 1,\ldots,N \}^\N$ be the collection of all infinite words obtained from integers $\{ 1,\ldots,N \}$. If $\iii = i_1i_2\cdots \in \Sigma$, then we define $\iii|_n = i_1 \cdots i_n$ for all $n \in \N$. The empty word $\iii|_0$ is denoted by $\varnothing$. Define $\Sigma_n = \{ \iii|_n : \iii \in \Sigma \}$ for all $n \in \N$ and $\Sigma_* = \bigcup_{n \in \N} \Sigma_n \cup \{ \varnothing \}$. Thus $\Sigma_*$ is the collection of all finite words. The length of $\iii \in \Sigma_* \cup \Sigma$ is denoted by $|\iii|$.
The concatenation of two words $\iii \in \Sigma_*$ and $\jjj \in \Sigma_* \cup \Sigma$ is denoted by $\iii\jjj$. Let $\sigma$ be the left shift operator defined by $\sigma\iii = i_2i_3\cdots$ for all $\iii = i_1i_2\cdots \in \Sigma$. If $\iii \in \Sigma_n$ for some $n$, then we set $[\iii] = \{ \jjj \in \Sigma : \jjj|_n = \iii \}$. The set $[\iii]$ is called a \emph{cylinder set}. The \emph{shift space} $(\Sigma,\sigma)$ is compact in the topology generated by the cylinder sets. Moreover, the cylinder sets are open and closed in this topology and they generate the Borel $\sigma$-algebra.

\subsection{Products of matrices}
Let $\mathbb{RP}^1$ be the real projective line, that is, the set of all lines through the origin in $\R^2$. Let $\mathsf{A} = (A_1,\ldots,A_N) \in GL_2(\R)^N$ and write $A_\iii = A_{i_1} \cdots A_{i_n}$ for all $\iii = i_1 \cdots i_n \in \Sigma_n$ and $n \in \N$. We say that $\A$ is \emph{irreducible} if there does not exist $V \in \RP$ such that $A_iV=V$ for all $i \in \{ 1,\ldots,N \}$; otherwise $\A$ is \emph{reducible}. The tuple $\A$ is \emph{strongly irreducible} if there does not exist a finite set $\mathcal{V} \subset \RP$ such that $A_i\mathcal{V}=\mathcal{V}$ for all $i\in\{1,\ldots,N\}$. In a reducible tuple $\A$, all the matrices are simultaneously upper triangular in some basis. For tuples with more than one element, strong irreducibility is a generic property.

We say that $\mathsf{A} = (A_1,\ldots,A_N) \in GL_2(\R)^\N$ is \emph{dominated} if there exist constants $C>0$ and $0<\tau<1$ such that
\begin{equation} \label{eq:domination}
  \frac{|\det(A_\iii)|}{\|A_\iii\|^2}\leq C \tau^{|\iii|}
\end{equation}
for all $\iii \in \Sigma_*$. We call a proper subset $\CC\subset\RP$ a \emph{multicone} if it is a finite union of closed projective intervals. We say that $\CC\subset\RP$ is a \emph{strongly invariant multicone} for $\mathsf{A}$ if it is a multicone and $A_i\CC\subset\CC^o$ for all $i \in \{1,\ldots,N\}$. Here $\CC^o$ is the interior of $\CC$. By \cite[Theorem~B]{BochiGourmelon2009}, $\mathsf{A}$ has a strongly invariant multicone if and only if $\mathsf{A}$ is dominated.

A matrix $A$ is called \emph{proximal} if it has two real eigenvalues with different absolute value, \emph{conformal} if it has two complex eigenvalues, and \emph{parabolic} if it is neither conformal nor proximal. If $\mathsf{A}$ is dominated, then, by \cite[Corollary 2.4]{BaranyKaenmakiMorris2018}, $\mathsf{A}$ contains only proximal elements. For a proximal matrix $A$, let $\lambda_u(A)$ and $\lambda_s(A)$ be the largest and smallest eigenvalues of $A$ in absolute value, respectively. Note that $|\lambda_u(A)|=\|A|u(A)\|$ and $|\lambda_s(A)|=\|A|s(A)\|$, where $u(A)\in\RP$ is the \emph{unstable direction}, i.e.\ the eigenspace of $A$ corresponding to $\lambda_u(A)$, and $s(A)\in\RP$ is the \emph{stable direction}, i.e.\ the eigenspace corresponding to $\lambda_s(A)$. In other words, $u(A) = \Ker(A-\lambda_u(A)I)$ and $s(A) = \Ker(A-\lambda_s(A)I)$. Let $\mathsf{A} = (A_1,\ldots,A_N) \in GL_2(\R)^N$ and $\mathcal{S}(\mathsf{A})$ be the subsemigroup of $GL_2(\mathbb{R})$ generated by $\mathsf{A}$. Note that $\overline{\mathbb{R}\mathcal{S}(\mathsf{A})}$, the closure of the set $\{cA : c\in\R \text{ and } A\in\mathcal{S}(\mathsf{A})\}$, is a subsemigroup of $M_2(\mathbb{R})$, the vector space of all $2 \times 2$ real matrices. Define
\begin{equation*}
  \mathscr{R}(\mathsf{A}) = \{ A \in \overline{\mathbb{R}\mathcal{S}(\mathsf{A})} : \rank(A)=1 \}.
\end{equation*}
Recall that, by \cite[Lemma 3.1]{BaranyKaenmakiMorris2018}, $\mathscr{R}(\mathsf{A}) \ne \emptyset$ if and only if $\mathcal{S}(\mathsf{A})$ contains at least one proximal or parabolic element. If $\mathscr{R}(\mathsf{A}) \ne \emptyset$, then we define
\begin{equation}\label{eq:YF}
  Y_F = \{V^{\perp} \in \mathbb{RP}^1 : V=A\R^2 \text{ for some } A \in \mathscr{R}(\mathsf{A}) \}.
\end{equation}
If $\mathsf{A}$ is dominated, then $Y_F$ is the closure of all possible orthogonal complements of the unstable directions $u(A)$ of proximal elements $A$ in $\mathcal{S}(\mathsf{A})$; see \cite[Lemma 3.4]{BaranyKaenmakiMorris2018}. Analogously, let $\mathcal{S}^{-1}(\mathsf{A})$ be the subsemigroup of $GL_2(\mathbb{R})$ generated by $(A_1^{-1},\ldots,A_N^{-1})$ and note that $\overline{\mathbb{R}\mathcal{S}^{-1}(\mathsf{A})}$ is a subsemigroup of $M_2(\mathbb{R})$. Define
\begin{equation*}
  \overleftarrow{\mathscr{R}}(\mathsf{A}) = \{ A \in \overline{\mathbb{R}\mathcal{S}^{-1}(\mathsf{A})} : \rank(A)=1 \}
\end{equation*}
and note that $\overleftarrow{\mathscr{R}}(\mathsf{A}) \ne \emptyset$ if and only if $\mathscr{R}(\mathsf{A}) \ne \emptyset$. If $\mathscr{R}(\mathsf{A}) \ne \emptyset$, then we define the set $X_F$ of all possible \emph{Furstenberg directions}, which is the closure of the unstable directions of the proximal and parabolic elements of $\mathcal{S}^{-1}(\mathsf{A})$, to be
\begin{equation*}
  X_F = \{V \in \mathbb{RP}^1 : V=A\R^2 \text{ for some } A \in \overleftarrow{\mathscr{R}}(\mathsf{A}) \}.
\end{equation*}

\begin{lemma}\label{lem:furstenberg}
  Let $\mathsf{A} = (A_1,\ldots,A_N) \in GL_2(\R)^N$ be such that $\mathscr{R}(\mathsf{A}) \ne \emptyset$. Then the closure of $\bigcup_{\iii\in\Sigma_*}A_{\iii}^{-1}Y_F$ contains $X_F$.
\end{lemma}

\begin{proof}
  Let $V\in X_F$. Then there exists $A\in\overleftarrow{\mathscr{R}}(\mathsf{A})$ such that $A\R^2=V$ and $\| A \| = 1$. By the definition of $\overleftarrow{\mathscr{R}}(\mathsf{A})$, there exists a sequence of matrices $B_n\in\mathcal{S}^{-1}(\mathsf{A})$ such that $B_n/\|B_n\|\to A$. Thus, $B_nW\to V$ as $n\to\infty$ for every $W\in\RP$ except possibly at most one. Therefore, if $\bigcup_{\iii\in\Sigma_*}A_{\iii}^{-1} Y_F$ contains also other points than this exceptional singleton, then the statement follows.

  Assume then that $\bigcup_{\iii\in\Sigma_*}A_{\iii}^{-1} Y_F=\{Z\}$. Then we have that all $A\in\mathsf{A}$ fix $Z$ and $B_n Z = \Ker(A)$ for all $n$. For each $n\in\N$ fix $L_n \in \mathbb{RP}^1$ be so that $\|B_n|L_n\| = \|B_n\|$. Then we have that $\| B_n | Z \| / \| B_n | L_n \| \to 0$ and so $\| B_n^{-1} |  Z \| / \| B_n^{-1} | B_n( L_n) \| \to \infty$. Thus, by passing to a subsequence if necessary, we have that $B_n^{-1}/\|B_n^{-1}\| \to B$ for which $B\R^2 = Z$ and so $Z^\perp \in Y_F$, which is a contradiction.
\end{proof}

\subsection{Lyapunov spectrum}
Let $\MM_\sigma(\Sigma)$ denote the collection of all $\sigma$-invariant Borel probability measures on $(\Sigma,\sigma)$. We say that a measure $\nu$ on $\Sigma$ is \emph{fully supported} if $\spt(\nu)=\Sigma$. A probability measure $\nu$ on $(\Sigma,\sigma)$ is \emph{Bernoulli} if there exist a probability vector $(p_1,\ldots,p_N)$ such that
$$
  \nu([\iii])=p_{i_1}\cdots p_{i_{n}}
$$
for all $\iii=i_1\cdots i_n\in\Sigma_n$ and $n \in \N$. It is well-known that Bernoulli measures are ergodic. We say that $\nu \in \MM_{\sigma^n}(\Sigma)$ is an \emph{$n$-step Bernoulli} if it is a Bernoulli measure on $(\Sigma,\sigma^n)$. Note that every Bernoulli measure is an $n$-step Bernoulli measure.

Let $A \in GL_2(\R)$ and $u(A^TA) = \Ker(A^TA-\|A\|I)$ be the eigenspace of $A^TA$ associated with $\|A\|$. Note that the singular values of $A$ are $\|A\| = \|A | u(A^TA)\|$ and $\|A^{-1}\|^{-1} = \|A^{-1} | u((A^{-1})^TA^{-1})\|^{-1}$. Let $\mathsf{A}=(A_1,\ldots,A_N) \in GL_2(\R)^N$ and define
\begin{align*}
  \vartheta_1(\iii|_n) &= A_{i_1} \cdots A_{i_n}(u((A_{i_1} \cdots A_{i_n})^TA_{i_1} \cdots A_{i_n})), \\
  \vartheta_2(\iii|_n) &= A_{i_1}^{-1} \cdots A^{-1}_{i_n}(u((A_{i_1}^{-1} \cdots A^{-1}_{i_n})^TA_{i_1}^{-1} \cdots A^{-1}_{i_n}))
\end{align*}
for all $\iii = i_1i_2\cdots \in \Sigma$ and $n \in \N$. By Kingman's Ergodic Theorem, it is well known that for each ergodic $\nu \in \MM_\sigma(\Sigma)$ there exist numbers $0 < \chi_1(\nu) \le \chi_2(\nu)$ such that
\begin{align*}
   \chi_1(\nu) &= -\lim_{n \to \infty} \tfrac{1}{n} \log \|A_{i_1} \cdots A_{i_n}\|, \\
   \chi_2(\nu) &= -\lim_{n \to \infty} \tfrac{1}{n} \log \|A_{i_1}^{-1} \cdots A^{-1}_{i_n}\|^{-1}
\end{align*}
for $\nu$-almost all $i_1i_2\cdots \in \Sigma$. The numbers are called \emph{Lyapunov exponents}. Another application of Kingman's Ergodic Theorem shows that
\begin{equation*}
  \chi_2(\nu) = -\lim_{n \to \infty} \tfrac{1}{n} \int_{\Sigma} \log \|A_{i_1}^{-1} \cdots A^{-1}_{i_n}\|^{-1} \dd\nu(\iii) = -\lim_{n \to \infty} \tfrac{1}{n} \log \|A_{i_n}^{-1} \cdots A^{-1}_{i_1}\|^{-1}
\end{equation*}
for $\nu$-almost all $i_1i_2\cdots \in \Sigma$. If $\chi_1(\nu) < \chi_2(\nu)$, then we say that $\nu$ has a \emph{simple Lyapunov spectrum}. Note that if $\mathsf{A}$ is dominated, then, by \eqref{eq:domination}, every measure has simple Lyapunov spectrum. It follows from \cite[Theorem 4.2]{MorrisShmerkin2019} that if $\nu$ has simple Lyapunov spectrum, then the semigroup $\mathcal{S}(\mathsf{A})$ contains a proximal element. Furthermore, by Oseledets' Theorem, the limit directions
\begin{align*}
  \vartheta_1(\iii) &= \lim_{n \to \infty}\vartheta_1(\iii|_n) \in \RP, \\
  \vartheta_2(\iii) &= \lim_{n \to \infty}\vartheta_2(\iii|_n) \in \RP
\end{align*}
exist for $\nu$-almost all $\iii \in \Sigma$. The measure $\mu_F = \vartheta_2\nu$ is called the \emph{Furstenberg measure} with respect to $\nu$. Here, if $\nu$ is a measure on $\Sigma$, then the pushforward measure of $\nu$ under a measurable map $f$ defined on $\Sigma$ is denoted by $f\nu$. Recall that, by \cite[Theorem~II.3.6]{BougerolLacroix1985}, for each $V\in\spt(\mu_F)\subset X_F$ it holds that
\begin{equation}\label{eq:furst}
  \chi_2(\nu) = -\lim_{n\to\infty}\tfrac{1}{n}\log\|A_{i_n}^{-1}\cdots A_{i_1}^{-1}|V\|^{-1}
\end{equation}
for $\nu$-almost all $i_1i_2\cdots \in \Sigma$. If $\nu$ is a Bernoulli measure and $(p_1,\ldots,p_N)$ the associated probability vector, then we call a measure $m$ on $\RP$ \emph{$\nu$-stationary} if
$$
  \int_{\RP} f(V)\dd m(V)=\sum_{i=1}^Np_i\int_{\RP} f(A_{i}^{-1}V)\dd m(V)
$$
for all continuous functions $f \colon \RP \to \R$. Observe that the Furstenberg measure $\mu_F$ is $\nu$-stationary. Let $T \colon \Sigma \times \RP \to \Sigma \times \RP$ be the skew-product defined by
\begin{equation} \label{eq:skew-product}
  T(\iii,V) = (\sigma\iii,A_{\iii|_1}^{-1} V).
\end{equation}
The measure $\nu \times \mu_F$ on $\Sigma \times \RP$ is $T$-invariant and ergodic; see e.g.\ \cite[Theorem~2.2]{BaranyKaenmaki2017}.

The following lemma characterizes what kind of matrix tuples there can be if we assume the existence of a Bernoulli measure having simple Lyapunov spectrum.

\begin{lemma}\label{lem:charac}
  Let $\mathsf{A} = (A_1,\ldots,A_N) \in GL_2(\R)^N$ and $\nu$ be a fully supported Bernoulli measure having simple Lyapunov spectrum. Then either $\mathsf{A}$ is strongly irreducible or there exists $M \in GL_2(\R)$ such that $MA_iA_jM^{-1}$ is upper-triangular for all $i,j \in \{1,\ldots,N\}$.
\end{lemma}

\begin{proof}
  Let us assume that $\mathsf{A}$ is not strongly irreducible. As $\nu$ has simple Lyapunov spectrum, the semigroup $\mathcal{S}(\mathsf{A})$ contains a proximal element. Thus, by \cite[Proposition~4.3]{BougerolLacroix1985}, there exists a subspace $V\in\RP$ such that $\{AV:A\in\mathcal{S}(\mathsf{A})\}$ contains at most two elements. If $\{AV:A\in\mathcal{S}(\mathsf{A})\}$ contains just one element, say, $\{AV:A\in\mathcal{S}(\mathsf{A})\}=\{W\}$, then clearly, $A_iW=W$ for all $i\in\{1,\ldots,N\}$. Hence, the matrices in $\mathsf{A}$ are simultaneously conjugated to upper-triangular matrices and we are done.

  If $\{AV:A\in\mathcal{S}(\mathsf{A})\}$ contains two elements, say, $\{AV:A\in\mathcal{S}(\mathsf{A})\}=\{W_1,W_2\}$, then clearly, $A_iW_j\in\{W_1,W_2\}$ for both $j \in \{1,2\}$ and for all $i\in\{1,\ldots,N\}$. Thus, by applying a change of coordinates, we may assume without loss of generality that
  $$
    A_i=
    \begin{pmatrix}
      a_i & 0 \\
      0 & b_i
    \end{pmatrix}\quad\text{or}\quad
    A_i=
    \begin{pmatrix}
      0 & a_i \\
      b_i & 0
    \end{pmatrix}
  $$
  for all $i\in\{1,\ldots,N\}$. If the matrices in $\mathsf{A}$ are all diagonal or they are all antidiagonal, then we are again done as the product of two antidiagonal matrices is diagonal. Thus, the only remaining case to be dealt with is the one where there exist $i\neq j$ such that
  $$
    A_i=
    \begin{pmatrix}
      a_i & 0 \\
      0 & b_i
    \end{pmatrix}\quad\text{and}\quad
    A_j=
    \begin{pmatrix}
      0 & a_j \\
      b_j & 0
    \end{pmatrix}.
  $$
  Let us first show that in this case $\mu_F=\frac{1}{2}(\delta_{e_1}+\delta_{e_2})$, where $\delta_{e_k}$ is the Dirac mass at the coordinate axis $e_k$. This follows by simple linear algebra if we can show that $\mu_F$ is atomic. Let $\mu_F'$ be the measure $\mu_F$ with all the atoms removed and suppose for a contradiction that $\mu_F'$ is non-trivial. Since $\mu_F'$ is non-atomic and $\nu$-stationary, we get, by \cite[Theorem~3.6(i)]{BougerolLacroix1985}, that $\lim_{n\to\infty}\sphericalangle(A_{i_n}\cdots A_{i_1}V,A_{i_n}\cdots A_{i_1}W)=0$ for all $V,W\in\RP$ and for $\nu$-almost all $\iii \in \Sigma$. Choosing $V=e_1$ and $W=e_2$, we see that this is not possible. Thus, $\mu_F'$ must be trivial and, consequently, $\mu_F=\frac{1}{2}(\delta_{e_1}+\delta_{e_2})$. Recalling that $\nu \times \mu_F$ is $T$-invariant and ergodic, Birkhoff Ergodic Theorem implies that
  \[
  \begin{split}
    \lim_{n\to\infty}\tfrac{1}{n}\log\|A_{i_n}^{-1}\cdots A_{i_1}^{-1}|e_j\|&=\tfrac{1}{2}\sum_{i=1}^Np_i(\log\|A_i^{-1}|e_1\|+\log\|A_i^{-1}|e_2\|)\\
    &=\tfrac{1}{2}\sum_{i=1}^Np_i\log|\det(A_i^{-1})|=\tfrac{1}{2}(\chi_1(\nu)+\chi_2(\nu)),
    \end{split}
  \]
  for $\nu$-almost all $\iii \in \Sigma$ and for both $j \in \{1,2\}$. Note that $\|A\| \in \{ \|A|e_1\|, \|A|e_2\| \}$ for any diagonal or antidiagonal matrix $A$. Therefore,
  \[
  \begin{split}
  \tfrac{1}{2}(\chi_1(\nu)+\chi_2(\nu))&<\chi_2(\nu)=\lim_{n\to\infty}\tfrac{1}{n}\log\|A_{i_n}^{-1}\cdots A_{i_1}^{-1}\| \\
  &\leq\lim_{n\to\infty} \max_{j \in \{1,2\}} \tfrac{1}{n}\log \|A_{i_n}^{-1}\cdots A_{i_1}^{-1}|e_j\| = \tfrac{1}{2}(\chi_1(\nu)+\chi_2(\nu))
  \end{split}
  \]
  which is a contradiction. Thus, the case in which $\mathsf{A}$ contains both diagonal and antidiagonal matrices is not possible.
\end{proof}

\begin{lemma}\label{lem:xfissupp}
  Let $\mathsf{A} = (A_1,\ldots,A_N) \in GL_2(\R)^N$ and $\nu$ be a fully supported Bernoulli measure having simple Lyapunov spectrum. Then there exist two fully supported $2$-step Bernoulli measures $\nu_1$ and $\nu_2$ having simple Lyapunov spectrum such that
  \begin{equation*}
    \spt(\mu_F^1)\cup\spt(\mu_F^2)=X_F,
  \end{equation*}
  where $\mu_F^i$ is the Furstenberg measure with respect to $\nu_i$.
\end{lemma}

\begin{proof}
  By Lemma \ref{lem:charac}, we may assume that the tuple $\mathsf{A}$ is either strongly irreducible or upper-triangular. Note that in the second case we have, to simplify notation, changed the base and replaced $\mathsf{A}$ by its second iterate, i.e.\ by the tuple consisting of all possible products $A_iA_j$.

  If $\mathsf{A}$ is strongly irreducible and $\mu_F$ is the Furstenberg measure with respect to a fully supported Bernoulli measure $\nu$ having simple Lyapunov spectrum, then we show that $\spt(\mu_F)=X_F$. Observe that, by \cite[Theorem~4.1]{BougerolLacroix1985}, $\mu_F$ is non-atomic. It suffices to show that $X_F \subset \spt(\mu_F)$, so let $V\in X_F$. Then there exists $A\in\overleftarrow{\mathscr{R}}(\mathsf{A})$ such that $A\R^2=V$. By definition of $\overleftarrow{\mathscr{R}}(\mathsf{A})$ there exists a sequence of matrices $B_n\in\mathcal{S}^{-1}(\mathsf{A})$ such that $B_n/\|B_n\|\to A$. Suppose for a contradiction that $V \notin \spt(\mu_F)$. Then, by the compactness of the support, there exists $\kappa>0$ such that $\mu_F(B(V,\kappa))=0$. Since $\mu_F$ is invariant and $\nu$ is fully supported, we have $B\mu_F(B(V,\kappa))=0$ for all $B\in\mathcal{S}^{-1}(\mathsf{A})$. But since the measure $\mu_F$ is non-atomic, \cite[Lemma~3.2]{BougerolLacroix1985} assures that $B_n\mu_F$ converges weakly to the Dirac mass $\delta_V$ at $V$, i.e.\ $B_n\mu_F(B(V,\kappa))\to1,$ which is a contradiction, so $V\in X_F$ implies $V\in \spt(\mu_F)$.

  Let us then assume that the tuple $\A=(A_1,\ldots,A_N)$ is upper-triangular. In this case, the inverse matrices are of the form
  \begin{equation*}
    A_i^{-1} =
    \begin{pmatrix}
      a_i & b_i \\
      0   & c_i
    \end{pmatrix}
  \end{equation*}
  for all $i \in \{1,\ldots,N\}$. There are three possible cases:
  \begin{enumerate}
    \item\label{it:I} There exist Bernoulli measures $\nu_1$ and $\nu_2$ such that
    \begin{equation*}
      \sum_{i=1}^N\nu_1([i])\log\biggl|\frac{a_i}{c_i}\biggr| < 0 < \sum_{i=1}^N\nu_2([i])\log\biggl|\frac{a_i}{c_i}\biggr|.
    \end{equation*}
    \item\label{it:II} For every Bernoulli measure $\nu_1$ it holds that
    \begin{equation*}
      \sum_{i=1}^N\nu_1([i])\log\biggl|\frac{a_i}{c_i}\biggr| \leq 0,
    \end{equation*}
    and there is at least one such Bernoulli measure for which the inequality holds strictly.
    \item\label{it:III} For every Bernoulli measure $\nu_2$ it holds that
    \begin{equation*}
      \sum_{i=1}^N\nu_2([i])\log\biggl|\frac{a_i}{c_i}\biggr| \geq 0.
    \end{equation*}
  \end{enumerate}
  It is easy to see that, in each of the cases, the Furstenberg measure $\mu_F^2$ with respect to $\nu_2$ is the Dirac mass $\delta_{e_1}$ at the $x$-coordinate axis $e_1$. Furthermore, the infinite series $u(\iii)=\sum_{n=1}^\infty \frac{b_{i_n}}{c_{i_n}}\prod_{k=1}^{n-1}\frac{a_{i_k}}{c_{i_k}}$ converges for $\nu_1$-almost every $\iii = i_1i_2\cdots \in \Sigma$, and
  \begin{equation*}
    A_{i_{1}}^{-1}(u(\sigma\iii),1) = c_{i_1}(u(\iii),1)
  \end{equation*}
  for $\nu_1$-almost every $\iii \in \Sigma$. Thus, the corresponding Furstenberg measure $\mu_F^1$ is the distribution of the subspaces $\linspan\{(u(\iii),1)\}$. We may assume that $\nu_1$ is fully supported.

  In the case \eqref{it:III}, we clearly have $X_F=\{e_1\}=\spt(\mu_F^2)$. In the cases \eqref{it:I} and \eqref{it:II}, if $\mu_F^1$ is non-atomic, then, by recalling that $\nu_1$ is fully supported, the argument in the case where $\mathsf{A}$ was strongly irreducible can be repeated. If $\mu_F^1$ is atomic, then, by using the invariance and uniqueness of $\mu_F^1$, we see that $\mu_F^1=\delta_V$ for some $e_1\neq V\in\RP$. Hence, in the case \eqref{it:II}, we have $X_F=\{V\}=\spt(\mu_F^1)$, and in the case \eqref{it:I}, we have $X_F=\{V,e_1\}=\spt(\mu_F^1)\cup\spt(\mu_F^2)$.
\end{proof}

\subsection{Self-affine set}
We consider a tuple $(A_1+v_1,\ldots,A_N+v_N)$ of contractive invertible affine self-maps on $\R^2$, where we have written $A+v$ to denote the affine map $x \mapsto Ax+v$ defined on $\R^2$ for all $2 \times 2$ matrices $A$ and translation vectors $v \in \R^2$. We also write $\fii_i = A_i+v_i$ for all $i \in \{1,\ldots,N\}$ and $\fii_\iii = \fii_{i_1} \circ \cdots \circ \fii_{i_n}$ for all $\iii = i_1 \cdots i_n \in \Sigma_n$ and $n \in \N$. Note that the associated tuple of matrices $(A_1,\ldots,A_N)$ is an element of $GL_2(\R)^N$ and satisfies $\max_{i \in \{1,\ldots,N\}}\|A_i\|<1$.

It is a classical result that there exists a unique non-empty compact set $X \subset \R^2$, called the \emph{self-affine set}, such that
\begin{equation*}
  X = \bigcup_{i=1}^N \fii_i(X).
\end{equation*}
The \emph{canonical projection} $\pi\colon \Sigma \to X$ is defined by $\pi\iii = \sum_{n=1}^\infty A_{\iii|_{n-1}} v_{i_n}$ for all $\iii = i_1i_2\cdots \in \Sigma$. It is easy to see that $\pi$ is continuous and $\pi\Sigma = X$. If $\nu$ is a Bernoulli measure, then its canonical projection $\pi\nu$ is the \emph{self-affine} measure on $X$. It is well known that a self-affine measure $\mu$ satisfies
\begin{equation} \label{eq:sa-invariant}
  \mu = \sum_{i=1}^N p_i\fii_i\mu,
\end{equation}
where $(p_1,\ldots,p_N)$ is the associated probability vector. By \cite[Theorem~2.4]{BaranyKaenmaki2017}, the local dimension of a self-affine measure $\mu$ exists for $\mu$-almost every point $x \in X$ and equals to $\dim(\mu)$, the upper/lower Hausdorff/packing dimension of $\mu$. Recall that $\dim(\mu) \le \dimh(X)$ for all self-affine measures $\mu$, where $\dimh(X)$ is the Hausdorff dimension of $X$.

We say that $X$ satisfies the \emph{strong separation condition} if $\fii_i(X) \cap \fii_j(X) = \emptyset$ whenever $i \ne j$. We also use convention that whenever we speak about a self-affine set $X$, then it is automatically accompanied with a tuple of affine maps which defines it. This makes it possible to write that e.g.\ ``$X$ is strongly irreducible'' which obviously then means that ``the corresponding tuple of matrices is strongly irreducible''.

\begin{proposition}\label{prop:app}
  Let $X$ be a dominated and strongly irreducible planar self-affine set satisfying the strong separation condition. Then for every $\eps>0$ there exists a fully supported $n$-step Bernoulli measure $\nu$ having simple Lyapunov spectrum such that
  \begin{equation*}
    \dim(\pi\nu) \ge \dimh(X)-\eps.
  \end{equation*}
\end{proposition}

\begin{proof}
  The \emph{Lyapunov dimension} of $\mu \in \MM_\sigma(\Sigma)$ is
  \begin{equation*}
    \diml(\mu) = \min\biggl\{\frac{h(\mu)}{\chi_1(\mu)},1+\frac{h(\mu)-\chi_1(\mu)}{\chi_2(\mu)}\biggr\},
  \end{equation*}
  where
  \begin{equation*}
    h(\mu) = \inf_{n \in \N} -\tfrac{1}{n} \sum_{\iii \in \Sigma_n} \mu([\iii]) \log\mu([\iii])
  \end{equation*}
  is the \emph{entropy} of $\mu$. By \cite[Theorem 2.9]{BaranyKaenmakiMorris2018} and \cite[Theorem A]{KaenmakiReeve2014}, there exists unique measure $\mu \in \MM_\sigma(\Sigma)$ such that there is a constant $C \ge 1$ for which
  \begin{equation*}
    C^{-1} \min\{\|A_\iii\|^s, |\det(A_\iii)|^{s-1}\|A_\iii\|^{2-s}\} \le \mu([\iii]) \le C \min\{\|A_\iii\|^s, |\det(A_\iii)|^{s-1}\|A_\iii\|^{2-s}\}
  \end{equation*}
  for all $\iii \in \Sigma_*$, where $s = \diml(\mu)$. Recalling \cite[Theorem 1.1]{barany2017hausdorff}, we thus have
  \begin{equation*}
    \dimh(X) = \min\{2,\diml(\mu)\}.
  \end{equation*}
  Fix $\eps>0$ and choose $\eps'>0$ such that $|\chi_1(\mu)-\chi_1(\nu)|<\eps'$, $|\chi_2(\mu)-\chi_2(\nu)|<\eps'$, and $|h(\mu)-h(\nu)|<\eps'$ imply
  \begin{equation*}
    |\diml(\mu)-\diml(\nu)|<\eps.
  \end{equation*}
  Let $n \in \N$ be such that $\tfrac{1}{n}\log\kappa^{-1} < \eps'$,
  \begin{align*}
    h(\mu) &\le -\tfrac{1}{n}\sum_{\iii\in\Sigma_n}\mu([\iii])\log\mu([\iii]) < h(\mu) + \eps', \\
    \chi_1(\mu)-\eps' &< -\tfrac{1}{n}\int_\Sigma \log\|A_{i_1} \cdots A_{i_n}\| \dd\mu(\iii) \le \chi_1(\mu), \\
    \chi_2(\mu) &\le -\tfrac{1}{n}\int_\Sigma \log\|A_{i_1}^{-1} \cdots A_{i_n}^{-1}\|^{-1} \dd\mu(\iii) < \chi_2(\mu) + \eps',
  \end{align*}
  and let $\nu$ be the $n$-step Bernoulli measure obtained from the probability vector $(\mu([\iii]))_{\iii\in\Sigma_n}$. Note that $\nu$ is clearly fully supported and, by \eqref{eq:domination}, $\nu$ has simple Lyapunov spectrum. Recall that, by \cite[Corollary 2.4]{BaranyKaenmakiMorris2018}, there exists a constant $0<\kappa<1$ such that $\|A_\iii A_\jjj\| \ge \kappa\|A_\iii\| \|A_\jjj\|$ for all $\iii, \jjj \in \Sigma_*$. Thus, we have
  \begin{align*}
    \chi_1(\mu)-\eps' &\le -\tfrac{1}{n}\int_\Sigma \log\|A_{i_1} \cdots A_{i_n}\| \dd\mu(\iii) \\
    &\le \sup_{k \in \N} -\tfrac{1}{kn}\sum_{\iii_1,\ldots,\iii_k \in \Sigma_n} \mu([\iii_1]) \cdots \mu([\iii_k]) \log\|A_{\iii_1 \cdots \iii_k}\| = \chi_1(\nu) \\
    &\le \lim_{k \to \infty} -\tfrac{1}{kn}\sum_{\iii_1,\ldots,\iii_k \in \Sigma_n} \mu([\iii_1]) \cdots \mu([\iii_k]) \log(\kappa^{k-1} \|A_{\iii_1}\| \cdots \|A_{\iii_k}\|) \\
    &= -\tfrac{1}{n} \sum_{\iii \in \Sigma_n} \mu([\iii]) \log\|A_\iii\| + \tfrac{1}{n}\log\kappa^{-1} \le \chi_1(\mu) + \eps'
  \end{align*}
  and, similarly, $\chi_2(\mu)+\eps' \ge \chi_2(\nu) \ge \chi_2(\mu)-\eps'$ and $h(\mu)+\eps' \ge h(\nu) \ge h(\mu)$. Therefore, by \cite[Theorem 1.2]{barany2017hausdorff},
  \begin{equation*}
    \dim(\pi\nu) = \min\{2,\diml(\nu)\} \ge \min\{2,\diml(\mu)\} - \eps = \dimh(X) - \eps
  \end{equation*}
  as required.
\end{proof}

\subsection{Projections and tangent sets}
Let $V,W \in \RP$ be such that $V \ne W$. The projection $\proj_V^W\colon \R^2 \to V$ along the line $W$ onto the line $V$ is the rank $1$ matrix defined by $\Ker(\proj^W_{V})=W$ and $\proj^W_{V}|_V=I|_V$. It is easy to see that $\|\proj^W_{V}\|\geq1$ and $\|\proj^W_{V}\|=1$ if and only if $V^\perp=W$. We denote the normalised projection by $\nproj_V^W$, that is, $\nproj_V^W=\|\proj^W_{V}\|^{-1}\proj^W_{V}$. To simplify notation, we denote the orthogonal projection $\proj_V^{V^\perp}$ by $\proj_V$. We say that a self-affine set $X$ satisfies the \emph{projection condition} if there exists $n_0\in\N$ such that $\proj_{V^{\perp}}X$ is a non-trivial closed line segment for all $V \in \bigcup_{n\geq n_0}\bigcup_{\iii\in\Sigma_n}A_{\iii}^{-1}Y_F$. Note that the projection condition implies the assumption (2) in \cite[Theorem~3.1]{KaenmakiKoivusaloRossi2017}. Therefore, \cite[Remark 3.4]{KaenmakiKoivusaloRossi2017} guarantees that $X$ is not contained in a line if it satisfies the projection condition.

Recall that a sequence $(T_n)_{n\in\N}$ of closed subsets of $B(0,1)$ converges to $T$ in Hausdorff distance if
\begin{equation*}
  \lim_{i \to \infty} \sup_{x \in T_i} \dist(x,T) = 0
\quad\text{and}\quad
  \lim_{i \to \infty} \sup_{y \in T} \dist(y,T_i) = 0.
\end{equation*}

\begin{lemma} \label{lem:projcond}
  Let $X$ be a planar self-affine set satisfying the strong separation condition and the projection condition. Then $\proj_{V^\perp}X$ is a closed line segment for all $V\in X_F$. In particular, $\proj_{V^\perp}\pi[\iii]$ is a closed line segment for all $\iii\in\Sigma_*$ and $V\in X_F$.
\end{lemma}

\begin{proof}
  Fix $V\in X_F$. Recalling Lemma~\ref{lem:furstenberg}, let $(\jjj_n)_{n \in \N}$ be a sequence in $\Sigma_*$ and $(V_n)_{n \in \N}$ a sequence in $Y_F$ such that $A_{\jjj_n}^{-1}V_n \to V$. Since $\proj_{(A_{\jjj_n}^{-1}V_n)^{\perp}}X$ is a line segment for every $n \geq n_0$, $X$ is not contained in a line, and the mapping $V \mapsto \proj_{V^{\perp}}X$ is continuous in Hausdorff distance, we see that also $\proj_{V^\perp}X$ is a closed line segment. Finally, since $\pi[\iii]=\fii_{\iii}(X)$, the set $\proj_{V^\perp}\pi[\iii]$ is an affine image of $\proj_{(A_{\iii}^{-1}V)^\perp}X$, and since $X_F$ is invariant with respect to the inverse matrices, also the last claim holds.
\end{proof}

A complementary concept to projections is that of slices. The following lemma shows that all the slices of self-affine sets have zero measure.

\begin{lemma} \label{lem:zero_slice}
  Let $X$ be a planar self-affine set satisfying the strong separation condition such that it is not contained in a non-trivial closed line segment. If $\nu$ is a fully supported Bernoulli measure, then $\pi\nu(L + x) = 0$ for all $L\in\RP$ and $x\in X$.
\end{lemma}

\begin{proof}
  Suppose, for a contradiction, that there exist a fully supported Bernoulli measure $\nu$, a line $L\in\RP$ and a point $x\in X$ such that $\pi\nu(L + x)>0$. Note that, by \eqref{eq:sa-invariant},
  \begin{equation} \label{eq:L-max}
     \pi\nu(\fii_\iii^{-1}(L+x)) = \sum_{\jjj\in\Sigma_n} p_{\jjj}\pi\nu(\fii_{\iii\jjj}^{-1}(L+x)) \le \max_{\jjj \in \Sigma_n} \pi\nu(\fii_{\iii\jjj}^{-1}(L+x))
  \end{equation}
  for all $\iii \in \Sigma_*$ and $n \in \N$. Here $p_\iii = p_{i_1} \cdots p_{i_n}$ for all $\iii = i_1 \cdots i_n \in \Sigma_n$. Therefore, we recursively find a sequence $(\iii_n)_{n\in\N}$ of words in $\Sigma_*$ such that $\iii_{n+1}|_n=\iii_n$ and
  \begin{equation} \label{eq:L-epsilon}
    0 < \pi\nu(L + x) \le \pi\nu(\varphi_{\iii_n}^{-1}(L + x))
  \end{equation}
  for all $n\in\N$.

  By the strong separation condition, $\pi\nu(\{\pi\iii\}) = \lim_{n \to \infty} \pi\nu(\fii_{\iii|_n}(X)) = \lim_{n \to \infty} p_{\iii|_n} = 0$ for all $\iii \in \Sigma$. Therefore, for each $n,m \in \N$ with $n \ne m$ exactly one of the two following conditions hold:
  \begin{enumerate}
    \item $\pi\nu(\varphi_{\iii_n}^{-1}(L + x) \cap \varphi_{\iii_m}^{-1}(L + x)) = 0$, \label{it:L1}
    \item $\varphi_{\iii_n}^{-1}(L + x) = \varphi_{\iii_m}^{-1}(L + x)$. \label{it:L2}
  \end{enumerate}
  Indeed, if \eqref{it:L1} does not hold, we have $A_{\iii_m}^{-1}L=A_{\iii_n}^{-1}L$, as otherwise the intersection is at most one point, and then, either $\varphi_{\iii_n}^{-1}(L + x) \cap \varphi_{\iii_m}^{-1}(L + x) = \emptyset$ or $\varphi_{\iii_n}^{-1}(L + x) = \varphi_{\iii_m}^{-1}(L + x)$. This observation together with \eqref{eq:L-epsilon} implies that the collection $\{\varphi_{\iii_n}^{-1}(L + x)\}_{n\in\N}$ is finite. Thus, there exists $n_0 \in \N$ such that $\varphi_{\iii_{n}}^{-1}(L + x) = \varphi_{\iii_{n_0}}^{-1}(L + x)$ for infinitely many $n\in\N$.

  Since $X$ is not contained in a line segment, there exists $\kkk\in\Sigma$ such that $\pi\kkk \notin \varphi_{\iii_{n_0}}^{-1}(L + x)$. By compactness, there exists $m_0\in\N$ such that $\varphi_{\kkk|_n}(X) \cap \varphi_{\iii_{n_0}}^{-1}(L + x) = \emptyset$ for all $n \geq m_0$. Fix $n \in \N$ such that $n - n_0 \ge m_0$ and $\varphi_{\iii_n}^{-1}(L + x) = \varphi_{\iii_{n_0}}^{-1}(L + x)$. With this choice, by \eqref{eq:L-epsilon} and \eqref{eq:L-max}, we have
  \begin{align*}
    0 &< \pi\nu(\varphi_{\iii_{n_0}}^{-1}(L + x)) = \sum_{\jjj \in \Sigma_{n-n_0} \setminus \{\kkk|_{n-n_0}\}} p_{\jjj} \pi\nu(\fii_{\iii|_{n_0}\jjj}^{-1}(L+x)) \\
    &\le (1-p_{\kkk|_{n-n_0}}) \pi\nu(\fii_{\iii_n}^{-1}(L+x)) < \pi\nu(\fii_{\iii_n}^{-1}(L+x)),
  \end{align*}
  which is a contradiction.
\end{proof}

Let $E \subset \R^2$ be compact. For each $x\in E$ and $r>0$ we define the \emph{magnification} $M_{x,r} \colon \R^2 \to \R^2$ by setting
\begin{equation*}
  M_{x,r}(z) = \frac{z-x}{r}
\end{equation*}
for all $z \in \R^2$. We say that $T$ is a \emph{tangent set} of $E$ at $x$ if there is a sequence $(r_n)_{n \in \N}$ of positive real numbers such that $\lim_{n \to \infty} r_n = 0$ and $M_{x,r_n}(E) \cap B(0,1) \to T$ in Hausdorff distance. We denote the collection of tangent sets of $E$ at $x$ by $\Tan(E,x)$. Furthermore, we say that $T$ is a \emph{weak tangent set} of $E$ if there exist sequences $(x_n)_{n\in\N}$ of points in $E$ and $(r_n)_{n \in \N}$ of positive real numbers such that $\lim_{n \to \infty} r_n = 0$ and $M_{x_n,r_n}(E) \cap B(0,1) \to T$ in Hausdorff distance. We denote the collection of weak tangent sets of $E$ by $\Tan(E)$. Note that $\bigcup_{x\in E}\Tan(E,x)\subset\Tan(E)$ where the inclusion can be strict.

A set $E \subset \R^2$ (or $E \subset \R$) is \emph{porous} if there exists $0<\alpha \le 1$ such that for every $x \in E$ and $0<r<\diam(E)$ there is a point $y \in E$ for which $B(y,\alpha r) \subset B(x,r) \setminus E$. We say that a set $E \subset \R^2$ is a \emph{comb} if there is a closed porous set $C \subset \R$ such that
\begin{equation*}
  E = (\R \times C) \cap B(0,1) \quad\text{or}\quad E = (\ell \times \{0\}) \cap B(0,1),
\end{equation*}
where $\ell$ is an interval containing at least one of the intervals $[-1,0]$ and $[0,1]$. Note that a comb is nowhere dense and of zero Lebesgue measure.

We observe that the following result of K\"aenm\"aki, Koivusalo, and Rossi \cite[Theorem~3.1]{KaenmakiKoivusaloRossi2017} is applicable in our setting.

\begin{theorem} \label{thm:KKR}
  Let $X$ be a planar self-affine set satisfying the strong separation condition and the projection condition. If $\nu$ is a Bernoulli measure having simple Lyapunov spectrum, then for $\nu$-almost every $\iii \in \Sigma$ and for each $T \in \Tan(X,\pi\iii)$ there exists a comb $C$ such that
  \begin{equation*}
    O_\iii T = C,
  \end{equation*}
  where $O_\iii$ is the rotation that takes $\vartheta_1(\iii)$ to the $x$-axis.
\end{theorem}

\subsection{Assouad dimension}

If $(Y,d)$ is a metric space, then the \emph{Assouad dimension} of a set $E \subset Y$, denoted by $\dima(E)$, is the infimum of all $s$ satisfying the following: There exists a constant $C \ge 1$ such that each set $E \cap B(x,R)$ can be covered by at most $C(R/r)^s$ balls of radius $r$ centered at $E$ for all $0<r<R$. It is easy to see that $\dimh(E) \le \dima(E)$ for all sets $E \subset Y$. For other basic properties of the Assouad dimension, see \cite{Fraser2014,Luukkainen1998}.

If $E \subset \R^2$ is compact, then it is straightforward to see that $\dimh(T) \le \dima(E)$ for all $T \in \Tan(E)$; see \cite[Proposition 6.1.5]{MackayTyson}. The following result of K\"aenm\"aki, Ojala, and Rossi \cite[Proposition 5.7]{KOR} shows that there exists a weak tangent set which attains the maximal possible value. The result introduces a way to calculate the Assouad dimension of a set by considering its weak tangents.

\begin{proposition} \label{thm:KOR}
  If $E \subset \R^2$ is compact, then $\dima(E) = \max\{\dimh(T) : T \in \Tan(E)\}$.
\end{proposition}

If $E \subset \R^2$, $(Y,d)$ is a metric space, and $\eta \colon [0,\infty) \to [0,\infty)$ is a homeomorphism, then a homeomorphism $f \colon E \to Y$ is \emph{$\eta$-quasisymmetric} if
\begin{equation*}
  \frac{d(f(x),f(y))}{d(f(x),f(z))} \le \eta\biggl( \frac{|x-y|}{|x-z|} \biggr)
\end{equation*}
for all $x,y,z \in E$ with $x \ne z$. Quasisymmetric mappings, introduced in \cite{BeurlingAhlfors1956,TukiaVaisala1980}, are a non-trivial generalization of bi-Lipschitz mappings. The \emph{conformal Assouad dimension} of $E$ is
\begin{equation*}
  \cdima(E) = \inf\{ \dima(E') : E' \text{ is a quasisymmetric image of } E \}.
\end{equation*}
It is worth emphasizing that the codomains of the quasisymmetric mappings used in the definition can be any metric spaces. The conformal Assouad dimension is bounded above by the Assouad dimension. A set $E \subset \R^2$ is \emph{minimal} for the conformal Assouad dimension if $\dima(E) = \cdima(E)$. We remark that conformal dimension and minimality can similarly be defined also for other set dimensions. The original definition is for the Hausdorff dimension and it was introduced in \cite{Pansu1989}. The conformal dimension measures the size of the best shape of $E$. Calculating the conformal dimension and characterizing minimality are challenging problems; for example, see \cite{BateOrponen2018,BonkMerenkov2013,Hakobyan2010,KOR,Mackay}.

\section{Main results} \label{sec:main-results}

Our first main theorem shows that generic points of a planar self-affine set share the tangent sets. More precisely, we will show that generic points share the projections of tangent sets in the direction of the ``comb teeth''. Relying on K\"aenm\"aki, Koivusalo, and Rossi \cite{KaenmakiKoivusaloRossi2017}, this characterizes uniquely the comb set except in the case, when this projection is a single point. Unfortunately, our method does not provide any information on the length of $\ell$ in the case $(\ell\times\{0\})\cap B(0,1)$, where $\ell$ is an interval containing at least one of the intervals $[-1, 0]$ and $[0, 1]$. Therefore, we identify all the combs of the form  $(\ell\times\{0\})\cap B(0,1)$.

\begin{theorem} \label{thm:maininformal}
	Let $X$ be a planar self-affine set satisfying the strong separation condition and the projection condition. If $\nu$ is a Bernoulli measure having simple Lyapunov spectrum, then there exists a collection $\CC$ of combs such that
	\begin{equation*}
		\{O_{\iii}T : T\in\Tan(X,\pi\iii)\} = \CC
	\end{equation*}
	for $\nu$-almost all $\iii \in \Sigma$, where $O_\iii$ is the rotation that takes $\vartheta_1(\iii)$ to the $x$-axis.
\end{theorem}

The theorem improves the results of K\"aenm\"aki, Koivusalo, and Rossi \cite[Theorem 3.1]{KaenmakiKoivusaloRossi2017} (see Theorem \ref{thm:KKR}) and Bandt and K\"aenm\"aki \cite[Theorems 1 and 2]{BandtKaenmaki2013}. The proof of the result can be found in Section \ref{sec:tangent-sets}.

Our second main result determines the Assouad dimension of a planar self-affine set and shows that the set is minimal for the conformal Assouad dimension. The theorem is proved in Section \ref{sec:assouad-dimension}.

\begin{theorem}\label{thm:assouad}
  Let $X$ be a planar self-affine set satisfying the strong separation condition and the projection condition. If for every $\eps>0$ there exists a fully supported $n$-step Bernoulli measure $\nu$ having simple Lyapunov spectrum such that $\dim(\pi\nu) \ge \dimh(X)-\eps$, then
  \begin{equation*}
    \cdima(X) = \dima(X)=1+\sup_{\atop{x\in X}{F\in X_F}}\dimh(X\cap(F+x)).
  \end{equation*}
\end{theorem}

Let us next investigate for which planar self-affine sets $X$ which satisfy the strong separation condition and the projection condition the above result can be applied to. The results of Bedford \cite{Bedford1984}, McMullen \cite{McMullen1984}, Gatzouras and Lalley \cite{gatzouraslalley}, and Bara\'nski \cite{Baranski2007} imply that Theorem~\ref{thm:assouad} is applicable for the self-affine sets described in these papers. Under the strong separation condition and the projection condition, our theorem thus generalizes the results of Jordan and Fraser \cite[Corollary~2.3(1)]{FraserJordan2017}, Mackay \cite[Theorems 1.1--1.4]{Mackay}, and K\"aenm\"aki, Ojala, and Rossi \cite[Theorem B]{KOR}. In fact, Theorem~\ref{thm:assouad} holds generically for self-affine sets $X$ defined by diagonal matrices: If the associated matrix tuple $\mathsf{A}$ is diagonal with $\max_{i \in \{1,\ldots,N\}}\|A_i\|<\tfrac12$, then, by \cite[Theorem 4.5]{Kaenmaki2004}, \cite[Theorem 1.7(iii)]{FengKaenmaki2011}, and the continuity of the Lyapunov exponents and the entropy on Bernoulli measures, Theorem~\ref{thm:assouad} is applicable for almost every choice of translation vectors.

In addition to diagonal systems, Theorem~\ref{thm:assouad} applies to self-affine sets $X$ defined by matrix tuples $\mathsf{A}$ consisting of simultaneously non-diagonalizable upper-triangular matrices having first diagonal element strictly larger than the second one; see B\'ar\'any, Hochman, and Rapaport \cite[Proposition~6.6]{barany2017hausdorff}. For other admissible triangular systems, see Bara\'nski \cite{baranski2} and Kolossv\'ary and Simon \cite{kolossvary2018triangular}. Finally, by Proposition \ref{prop:app}, Theorem~\ref{thm:assouad} is also applicable if the associated matrix tuple $\mathsf{A}$ is dominated and strongly irreducible:

\begin{example}\label{example}
  We exhibit a dominated and strongly irreducible planar self-affine set $X$ satisfying the strong separation condition and the projection condition. By Proposition \ref{prop:app} and Theorem~\ref{thm:assouad}, we thus have $\cdima(X) = \dima(X)=1+\sup\{\dimh(X\cap(F+x)) : x\in X \text{ and } F\in X_F\}$.

  Fix
  \begin{alignat*}{3}
    \textrm{(i)}\;\,   & 0<\gamma<\tfrac12<\lambda<1, \qquad\qquad\quad &
    \textrm{(ii)}\;\,  & a,b,d>0,                     \\
    \textrm{(iii)}\;\, & 0<v_3^1<1-a-b,               \qquad\qquad\quad &
    \textrm{(iv)}\;\,  & \gamma<v_3^2<1-\gamma-b-d,   \\
    \textrm{(v)}\;\,   & ad>b^2,                      \qquad\qquad\quad &
    \textrm{(vi)}\;\,  & \frac{1-2\lambda}{1-2\gamma}<\frac{a-d-\sqrt{(a-d)^2+4b^2}}{2b},
  \end{alignat*}
  and define
  \begin{equation*}
    A_1 =
    \begin{pmatrix}
      \lambda & 0 \\
      0 & \gamma
    \end{pmatrix}, \quad
    A_2 =
    \begin{pmatrix}
      \lambda & 0 \\
      0 & \gamma
    \end{pmatrix}, \quad
    A_3 =
    \begin{pmatrix}
      a & b \\
      b & d
    \end{pmatrix},
  \end{equation*}
  and
  \begin{align*}
    \fii_1(x,y) &= A_1(x,y), \\
    \fii_2(x,y) &= A_2(x,y) + (1-\lambda,1-\gamma), \\
    \fii_3(x,y) &= A_3(x,y) + (v_3^1,v_3^2)
  \end{align*}
  for all $(x,y) \in \R^2$. Let $X \subset [0,1]^2$ be the self-affine set associated to the tuple $(\fii_1,\fii_2,\fii_3)$; see Figure \ref{fig:example} for illustration.
  \begin{figure}
  \begin{tikzpicture}[scale=0.7]
    \draw[fill=black!10!white, draw=black] (0,2.5) -- (7.5,2.5) -- (7.5,0) -- (0,0) -- cycle;
    \draw[fill=black!10!white, draw=black] (2.5,10) -- (2.5,7.5) -- (10,7.5) -- (10,10) -- cycle;
    \draw[fill=black!10!white, draw=black] (2,3) -- (6,4) -- (7,7) -- (3,6) -- cycle;

    \draw (0,2.5) -- (0,10) -- (2.5,10);
    \draw (7.5,0) -- (10,0) -- (10,7.5);

    \node[align=center] at (3.75,1.25) {$\fii_1([0,1]^2)$};
    \node[align=center] at (6.25,8.75) {$\fii_2([0,1]^2)$};
    \node[align=center] at (4.5,5) {$\fii_3([0,1]^2)$};

    \node[align=center] at (3.75,-0.4) {$\lambda$};
    \node[align=center] at (-0.4,1.25) {$\gamma$};
    \node[align=center] at (6.25,10.4) {$\lambda$};
    \node[align=center] at (10.4,8.75) {$\gamma$};
  \end{tikzpicture}
  \caption{Images of the unit square under the mappings $\fii_i$ in Example \ref{example}.}
  \label{fig:example}
  \end{figure}
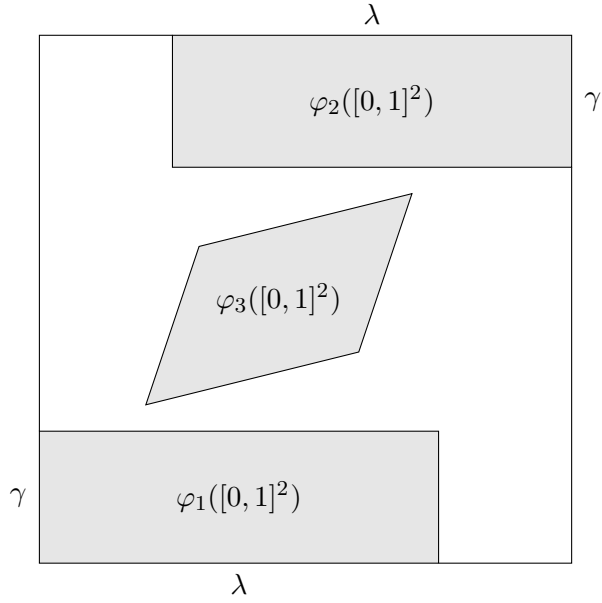

  Let us first show that the tuple $\mathsf{A} = (A_1,A_2,A_3)$ is dominated and strongly irreducible. Since $A_1$ and $A_2$ preserve only the $x$-axis and $y$-axis and, by (ii), $A_3$ does not preserve neither axis nor their union, it follows that $\mathsf{A}$ is strongly irreducible. By (v), the eigendirections of the symmetric matrices $A_1$, $A_2$, and $A_3$ corresponding to the largest eigenvalues in absolute value are
  \begin{equation*}
    u_1 = u_2 = (1,0) \quad \text{and} \quad u_3 = \biggl(\frac{a - d + \sqrt{(a-d)^2 + 4 b^2}}{2 b},1\biggr),
  \end{equation*}
  respectively. Thus there exist $\eps>0$ such that the closed $\eps$-neighborhood of
  \begin{equation*}
    \{ \linspan\{tu_1 + (1-t)u_3\} \in \RP : 0 \le t \le 1 \}
  \end{equation*}
  gets mapped into its interior by all the matrices. Therefore, $\mathsf{A}$ is dominated.

  Let us then show that $X$ satisfies the strong separation condition and the projection condition. The conditions (i)--(iv) guarantee that $\fii_i([0,1]^2) \cap \fii_j([0,1]^2) = \emptyset$ whenever $i \ne j$ implying the strong separation condition. To see that the projection condition holds, recall that, by the definition of $Y_F$ in \eqref{eq:YF},
  \begin{equation*}
    \bigcup_{\iii \in \bigcup_{n \in \N}\{1,2,3\}^n} A_\iii^{-1} Y_F \subset \{ \linspan\{(t,1)\} \in \RP : \frac{a-d-\sqrt{(a-d)^2+4b^2}}{2b} \le t \le 0 \}.
  \end{equation*}
  Let $Z$ be the self-affine set associated to the tuple $(\fii_1,\fii_2)$. By (vi), it is enough to show that $\proj_{\linspan\{(t,1)\}^\bot}(Z)$ is a non-trivial closed line segment for all $\frac{1-2\lambda}{1-2\gamma} < t \le 0$. Up to a linear transformation, it is enough to show that $P_t(Z) = [0,1-t]$ for all $\frac{1-2\lambda}{1-2\gamma} < t \le 0$, where $P_t(x,y) = x-ty$ for all $(x,y) \in \R^2$. To see this, fix $\frac{1-2\lambda}{1-2\gamma} < t \le 0$ and notice that it is sufficient to prove
  \begin{equation} \label{eq:example1}
    P_t(\varphi_{\iii}([0,1]^2)) = \bigcup_{i\in\{1,2\}}P_t(\varphi_{\iii i}([0,1]^2))
  \end{equation}
  for all $\iii \in \bigcup_{n \in \N} \{1,2\}^n$. Indeed, if \eqref{eq:example1} holds, then
  \begin{equation*}
    P_t(Z) = P_t\biggl(\bigcap_{n=1}^\infty\bigcup_{\iii\in\{1,2\}^n}\varphi_\iii([0,1]^2)\biggr) = \bigcap_{n=1}^\infty\bigcup_{\iii\in\{1,2\}^n}P_t(\varphi_\iii([0,1]^2))=P_t([0,1]^2) = [0,1-t].
  \end{equation*}
  To prove \eqref{eq:example1}, fix $n \in \N$ and $\iii = i_1 \cdots i_n \in \{1,2\}^n$. Write $v_1=(0,0)$ and $v_2=(1-\lambda,1-\gamma)$, and observe that
  \begin{align*}
    \fii_\iii([0,1]^2) &= A_\iii([0,1]^2) + \fii_\iii(0) = A_1^n([0,1]^2) + \sum_{k=1}^n A_{\iii|_{k-1}}v_{i_n} \\
    &= [0,\lambda^n] \times [0,\gamma^n] + \biggl(\sum_{k=1}^n v_{i_k}^1 \lambda^{k-1}, \sum_{k=1}^n v_{i_k}^2 \gamma^{k-1}\biggr).
  \end{align*}
  Therefore,
  \begin{align*}
    P_t(\fii_\iii([0,1]^2)) &= P_t\biggl(\biggl([0,\lambda^n] + \sum_{k=1}^n v_{i_k}^1 \lambda^{k-1}\biggr) \times \biggl([0,\gamma^n] + \sum_{k=1}^n v_{i_k}^2 \gamma^{k-1}\biggr)\biggr) \\
    &= \sum_{k=1}^n v_{i_k}^1 \lambda^{k-1} - t\sum_{k=1}^n v_{i_k}^2 \gamma^{k-1} + [0,\lambda^n-t\gamma^n].
  \end{align*}
  Similarly, we see that
  \begin{equation*}
    P_t(\fii_{\iii 1}([0,1]^2)) = \sum_{k=1}^n v_{i_k}^1 \lambda^{k-1} - t\sum_{k=1}^n v_{i_k}^2 \gamma^{k-1} + [0,\lambda^{n+1}-t\gamma^{n+1}]
  \end{equation*}
  and
  \begin{equation*}
    P_t(\fii_{\iii 2}([0,1]^2)) = \sum_{k=1}^n v_{i_k}^1 \lambda^{k-1} + (1-\lambda)\lambda^n - t\sum_{k=1}^n v_{i_k}^2 \gamma^{k-1} - t(1-\gamma)\gamma^n + [0,\lambda^{n+1}-t\gamma^{n+1}].
  \end{equation*}
  Thus, $P_t(\varphi_{\iii}([0,1]^2)) = P_t(\varphi_{\iii 1}([0,1]^2)) \cup P_t(\varphi_{\iii 2}([0,1]^2))$ if and only if $(1-\lambda)\lambda^n - t(1-\gamma)\gamma^n \le \lambda^{n+1}-t\gamma^{n+1}$ or, equivalently, $\frac{\lambda^n(1-2\lambda)}{\gamma^n(1-2\gamma)} \le t$. Since $\frac{1-2\lambda}{1-2\gamma} < t \le 0$, this holds for all $n \in \N$ by (i).
\end{example}

\section{Tangent sets} \label{sec:tangent-sets}

In this section, we develop the machinery to study the tangential structure of self-affine sets and prove Theorem \ref{thm:maininformal}. We begin with the following Poincar\'e recurrence type lemma which allows us to approximate a given point $(\iii,F)$ arbitrary well by the elements of the orbit of a generic point $(\jjj,L)$ under the dynamics introduced by the skew-product $T \colon \Sigma \times \RP \to \Sigma \times \RP$ defined in \eqref{eq:skew-product}.

\begin{lemma}\label{lem:visit}
  Let $\mathsf{A}=(A_1,\ldots,A_N) \in GL_2(\R)^N$ and $\nu$ be a fully supported Bernoulli measure having simple Lyapunov spectrum. If $\mu_F$ is the Furstenberg measure with respect to $\nu$, then there exists a set $\Theta \subset \Sigma \times \RP$ with $\nu \times \mu_F(\Theta)=1$ such that $\vartheta_1(\jjj)$ exists,
  \begin{align*}
    \chi_1(\nu) &= -\lim_{n\to\infty} \tfrac{1}{n}\log\|A_{\jjj|_n}^{-1}|\vartheta_1(\jjj)\|^{-1}, \\
    \chi_2(\nu) &= -\lim_{n\to\infty} \tfrac{1}{n}\log\|A_{\jjj|_n}^{-1}|L\|^{-1}
  \end{align*}
  for all $(\jjj,L) \in \Theta$, and there exists a countable dense set $\mathcal{F}$ so that
  \begin{equation*}
    \#\{n\in\N : T^n(\jjj,L)\in[\iii]\times B(F,\tfrac{1}{k})\} = \infty
  \end{equation*}
  for all $(\jjj,L) \in \Theta$, $\iii \in \Sigma_*$, $F \in \mathcal{F}$, and $k \in \N$.
\end{lemma}

\begin{proof}
  Since $\spt(\mu_F)$ is a compact subset of $\RP$, such a countable dense set $\mathcal{F} \subset \spt(\mu_F)$ exists. Note that $\nu \times \mu_F([\iii]\times B(F,r))>0$ for all $\iii\in\Sigma_*$, $F\in\mathcal{F}$, and $r>0$. The existence of the claimed set $\Theta \subset \Sigma\times\RP$ now follows from Oseledets' Theorem, Birkhoff Ergodic Theorem, and countability of the parameters $\iii \in \Sigma_*$, $F \in \mathcal{F}$, and $k \in \N$.
\end{proof}

Observe that the statement of Lemma \ref{lem:visit} actually holds for every $F \in \spt(\mu_F)$. We stated the lemma in the form we apply it later. We shall often use a line $L \subset \R^2$ to divide the plane into two half-planes (in the obvious way). In such situations we refer to these half-planes as \emph{half-planes determined by $L$}.

\begin{lemma}\label{lem:omega2}
	Let $X$ be a planar self-affine set satisfying the strong separation condition and the projection condition. If $\nu$ is a fully supported Bernoulli measure, then the set
	\begin{align*}
		B_\iii = \{F\in\RP : \;&\text{there exist }\kkk\in\Sigma \text{ and } n\in\N \text{ such that} \\
    &\pi\kkk\in F+\pi\iii \text{ and } \pi[\kkk|_n] \text{ is contained in one} \\
    &\text{of the closed half-planes determined by }F+\pi\iii\}
	\end{align*}
	 is at most countable for $\nu$-almost all $\iii\in\Sigma$.
\end{lemma}

\begin{proof}
  Let $C = \conv(X)$ be the convex hull of $X$ and denote its interior by $C^o$. Let us first show that $C^o\cap X\neq\emptyset$. Suppose to the contrary that $C^o \cap X = \emptyset$. Then $X \subset \partial C := C\setminus C^o$. By \cite[Remark 3.4]{KaenmakiKoivusaloRossi2017}, we know that $X$ is not contained in a line and hence, $\partial C$ is not a line segment. Let $V \in \bigcup_{n\geq n_0} \bigcup_{\iii\in\Sigma_n} A_\iii^{-1}Y_F$. Since the projection condition holds, $\proj_{V^\perp}X$ is a non-trivial closed line segment and hence, there exist $x_1,x_2 \in \partial C$ such that $x_1-x_2\in V$, the line segment connecting $x_1$ and $x_2$ is not contained in $\partial C$, and $x_1\in X$ or $x_2\in X$. Without loss of generality, we may assume that $x_1\in X$.

  There are now two cases, either $x_2 \in X$ or $x_2 \notin X$. If $x_2 \notin X$, then there exists an open arc $J$ such that $x_2\in J$ and $J\subset \partial C\setminus X$, and whose endpoints $w_1,w_2 \in \R^2$ are contained in $X$ but $w_1-w_2 \notin V$. Since $\{x_1\} = \bigcap_{n \in \N} \fii_{\iii|_n}(X)$ for some $\iii \in \Sigma$, there exists $n \in \N$ such that $x_1\in\varphi_{\iii|_n}(X)$ and $V+x$ intersects $J$ for all $x\in\varphi_{\iii|_n}(\partial C)\cap\partial C$. Observe that $\varphi_{\iii|_n}(X) \subset \partial C$ is a self-affine set associated to the tuple $(\varphi_{\iii|_n} \circ \varphi_1 \circ \varphi_{\iii|_n}^{-1},\ldots,\varphi_{\iii|_n} \circ \varphi_N \circ \varphi_{\iii|_n}^{-1})$ and satisfies the strong separation condition. Thus, $\varphi_{\iii|_n}(X)$ cannot be a closed arc segment. Indeed, if $\varphi_{\iii|_n}(X)$ was a closed arc, then it is a union of closed arcs $\varphi_{\iii|_n}(\varphi_i(X))$ which, by the strong separation condition, are pairwise disjoint, a contradiction. But this means that there exists an open arc $I\subset\varphi_{\iii|_n}(\partial C)\cap\partial C$ such that $X\cap I=\emptyset$, the endpoints $v_1,v_2 \in \R^2$ of $I$ are in $X$ and $v_1-v_2 \notin V$, and $V+x$ intersects $J$ for all $x\in I$. Thus, $\proj_{V^\perp}(X)\subset\proj_{V^\perp}(\partial C\setminus(I\cup J))=\proj_{V^\perp}\partial C\setminus(\proj_{V^\perp}I\cap\proj_{V^\perp}J)$, where $\proj_{V^\perp}I\cap\proj_{V^\perp}J$ is clearly non-empty and open, and have both endpoints in $\proj_{V^\perp}X$ contradicting the projection condition. Therefore, $C^o\cap X\neq\emptyset$ provided that $x_2 \notin X$.

  If $x_2\in X$, then, as $\{x_2\} = \bigcap_{m \in \N} \fii_{\jjj|_m}(X)$ for some $\jjj \in \Sigma$, there exists $m \in \N$ such that $\proj_{V^\perp}\varphi_{\jjj|_m}(X)\subset(\proj_{V^\perp}X)^o$ and $x_1 \notin \varphi_{\jjj|_m}(X)$. Again $\varphi_{\jjj|_m}(X)$ cannot be a closed arc and thus, there exists an open arc $J\subset\varphi_{\jjj|_m}(\partial C) \cap \partial C$ such that $J\cap X = \emptyset$ and the endpoints $w_1,w_2 \in \R^2$ of $J$ are contained in $X$ so that $w_1-w_2 \notin V$. By the projection condition, for any $y\in J$, $V+y$ must intersect $X$. By replacing $x_2$ and $x_1$ by an arbitrary point in $J$ and $V+x_2\cap X$, respectively, we can repeat the previous argument. Thus, $C^o\cap X\neq\emptyset$ also in this case.

  Let us next show that if $\nu$ is a fully supported Bernoulli measure, then $\pi\nu(\partial C \cap X) = 0$. It is easy to see that $\partial C \cap X \subset \bigcup_i\varphi_{i=1}^N(\partial C\cap X)$, i.e.\ $\partial C \cap X$ is sub-self-affine. Since $C^o \cap X \neq \emptyset$, there exists $\jjj \in \Sigma_*$ such that $\varphi_\jjj(X) \cap \partial C = \emptyset$. Since $\nu$ is fully supported, Birkhoff Ergodic Theorem implies that $\sigma^n\iii\in[\jjj]$ infinitely often for $\nu$-almost all $\iii \in \Sigma$. Therefore, $\pi\nu(\partial C\cap X)=0$ and, consequently, $\pi\nu(\bigcup_{\iii\in\Sigma_*}\varphi_\iii(\partial C))=0$. Observe that for every $\iii \in \Sigma$ with $\pi\iii \in X \setminus \bigcup_{\iii\in\Sigma_*}\varphi_\iii(\partial C)$ and for every $F \in \RP$ and $n \in \N$, the cylinder $\pi[\iii|_n]$ intersects both the half-planes determined by $F+\pi(\iii)$. The claim then follows, since there exist only countably many cylinders and for each canonical projection of a cylinder there are at most two lines $F \in \RP$ that meet the cylinder so that the cylinder is contained in only one of the closed half-planes determined by $F+\pi\iii$.
\end{proof}

The core of the following lemma is that, generically, the lines $L$ suitable for the previous Poincar\'e recurrence result are such that a generic limit direction $\vartheta_1$ stays away from it. This allows us to consider projections in coordinates given by these lines.

\begin{lemma}\label{lem:omega}
  Let $\mathsf{A}=(A_1,\ldots,A_N) \in GL_2(\R)^N$ and $\nu$ be a fully supported Bernoulli measure having simple Lyapunov spectrum. If $\mu_F$ is the Furstenberg measure with respect to $\nu$, then there exists a set $\Omega \subset \Sigma$ with $\nu(\Omega)=1$ such that for every $\iii \in \Omega$ we have $\mu_F(B_\iii) = 0$, and for almost every $L \in \spt(\mu_F) \setminus B_\iii$ it holds that
  \begin{equation*}
    \sphericalangle(\vartheta_1(\iii),L) > 0 \quad\text{and}\quad (\iii,L) \in \Theta,
  \end{equation*}
  where the sets $B_\iii$ are as in Lemma \ref{lem:omega2} and $\Theta \subset \Sigma \times \RP$ is as in Lemma \ref{lem:visit}.
\end{lemma}

\begin{proof}
  Let $\mathbb{B}$ be the set of atoms of $\mu_F$. Notice that $\mathbb{B}$ is at most countable and can be empty. For each $B\in\mathbb{B}$, we write $I_B = \{ \iii\in\Sigma : B\in B_\iii \}$ and choose $x\in\R^2$ to be such that $X$ is contained in the interior of only one of the half-planes determined by $B+x$. Furthermore, for each $\kkk\in\Sigma_*$ let $\underline{\kkk}, \overline{\kkk} \in [\kkk]$ be the closest and farthest points of $[\kkk]$ from the line $B + x$, respectively. That is, $\dist(B+x,\pi[\kkk])=\dist(B+x , \pi\underline{\kkk})$ and $\max_{\jjj\in[\kkk]}\dist(B+x,\pi\jjj)=\dist(B+x , \pi\overline{\kkk})$. It follows that $I_B=\bigcup_{\kkk\in\Sigma_*}((B+\pi\underline{\kkk})\cap X)\cup((B+\pi\overline{\kkk})\cap X)$. By Lemma \ref{lem:zero_slice}, $\pi\nu((B+\pi\underline{\kkk}) \cup (B+\pi\overline{\kkk}))=0$, and since $\mathbb{B}$ is at most countable, we have that
  \[
   \nu\biggl( \bigcup_{B\in\mathbb{B}} I_B  \biggr) = 0.
  \]
  Let $\Theta$ be as in Lemma~\ref{lem:visit}. By Fubini's Theorem, there exist measurable sets $\Omega'\subset\Sigma$ and $\Theta'_\iii\subset\RP$ such that
  $$
    \nu\times\mu_F\biggl(\Theta\triangle\bigcup_{\iii\in\Omega'}\{\iii\}\times\Theta_\iii'\biggr)=0
  $$
  and $\nu(\Omega')=1=\mu_F(\Theta_\iii')$ for every $\iii\in\Omega'$. Here $A \triangle B$ is the symmetric difference of the sets $A$ and $B$. By setting $\Omega = \Omega' \setminus \bigcup_{B\in\mathbb{B}} I_B$, we have $\mu_F(B_\iii)=0$ for all $\iii\in\Omega$. So we set $\Theta_\iii=\Theta_\iii'\setminus B_\iii$.

  Fix $\iii\in\Omega$ and observe that for each $\eps$ there exists $\roo>0$ depending on $\iii$ such that $\mu_F(\RP\setminus B(\vartheta_1(\iii),\roo))>1-\eps$. Indeed, if this is not the case, then $\mu_F(\{\vartheta_1(\iii)\}) > 0$. By Lemma \ref{lem:visit},
  $$
    -\lim_{n\to\infty}\tfrac{1}{n}\log\|A_{\iii|_n}^{-1}|L\|^{-1} = \chi_2(\nu) > \chi_1(\nu) = -\lim_{n\to\infty}\tfrac{1}{n}\log\|A_{\iii|_n}^{-1}|\vartheta_1(\iii)\|^{-1},
  $$
  for all $L\in\Theta_\iii$. Since $\mu_F(\Theta_\iii)=1$ and $\mu_F(\{\vartheta_1(\iii)\}) > 0$, we have $L = \vartheta_1(\iii)$ for some $L \in \Theta_\iii$ and putting this to the above inequality gives a contradiction. Thus, by choosing $\roo>0$ such that $\mu_F(\RP\setminus B(\vartheta_1(\iii),\roo))>1-\eps$, $\mu_F$-almost every $L \in (\RP\setminus B(\vartheta_1(\iii),\roo) ) \cap \Theta_\jjj$ suffices for the claim. Letting $\eps\downarrow 0$ finishes the proof.
\end{proof}

The following technical result provides us with necessary tools to study the tangent structure of the self-affine set. It shows that the magnifications of the affine mappings converge to projections along the limit direction $\vartheta_1$.

\begin{proposition}\label{prop:maintech}
  Let $X$ be a planar self-affine set satisfying the strong separation condition and $\nu$ be a fully supported Bernoulli measure having simple Lyapunov spectrum. Then there exists a set $\Omega \subset \Sigma$ with $\nu(\Omega)=1$ such that for every $\iii\in\Sigma$, $\jjj \in \Omega$, $F\in\spt(\mu_F)\setminus B_\jjj$ with $\sphericalangle(\vartheta_1(\jjj),F)>0$, and for every decreasing sequence $(r_k)_{k\in\N}$ of positive real numbers with $\lim_{k\to\infty}r_k \le \min_{i\neq j}\dist(\varphi_i(X),\varphi_j(X))/2$ there exist a sequence $(s_k)_{k\in\N}$ of positive real numbers with $\lim_{k\to\infty}s_k=0$ and a sequence $(n_k)_{k\in\N}$ of natural numbers with $\lim_{k\to\infty}n_k=\infty$ such that
  \begin{align*}
    X \cap B(\pi\jjj,s_k) &\subset \varphi_{\jjj|_{n_k}}(X), \\
    \varphi_{\jjj|_{n_k}}^{-1}(B(\pi\jjj,s_k)) &\subset B(\pi\iii,r_k)
  \end{align*}
  for all $k\in\N$. Furthermore, for any accumulation point $G$ of the sequence $(M_{\pi\iii,r_k} \circ \fii_{\jjj|_{n_k}}^{-1} \circ M^{-1}_{\pi\jjj,s_k})_{k\in\N}$ of affine maps we have $G\in\{\nproj_F^{\vartheta_1(\jjj)},-\nproj_F^{\vartheta_1(\jjj)}\}$.
\end{proposition}

\begin{proof}
  Let $\Omega \subset \Sigma$ be as in Lemma~\ref{lem:omega} and let $\Theta \subset \Sigma\times\RP$ and $\mathcal{F} \subset \RP$ be as in Lemma \ref{lem:visit}. Fix $\iii\in\Sigma$, $\jjj \in \Omega$, and $F\in\spt(\mu_F)$, and let the sequence $(r_k)_{k\in\N}$ be as in the formulation. Recalling Lemma~\ref{lem:omega}, let $L\in\spt(\mu_F)$ be such that $\sphericalangle(\vartheta_1(\jjj),L)>0$ and $(\jjj,L)\in\Theta$. Moreover, let $(F_k)_{k\in\N}$ be a sequence of elements in $\mathcal{F}$ such that $\sphericalangle(F_k,F)<\tfrac{1}{k}$ for all $k\in\N$.

  Relying on Lemma~\ref{lem:visit}, we find a sequence $(n_k)_{k\in\N}$ such that $n_k\to\infty$ as $k\to\infty$,
  \begin{align}
    |\fii_{\jjj|_{n_k}}^{-1}(\pi\jjj) - \pi\iii| &= |\pi\sigma^{n_k}\jjj - \pi\iii| < 2^{-k}r_k, \label{eq:conv1} \\
    A_{\jjj|_{n_k}}^{-1} L &\in B(F_k,\tfrac{1}{k}). \label{eq:conv2}
  \end{align}
  For each $k\in\N$, we choose $s_k = \max\{ s : \fii_{\jjj|_{n_k}}^{-1}(B(\pi\jjj,s)) \subset B(\pi\iii,r_k) \}$ and write $G_k = M_{\pi\iii,r_k} \circ \fii_{\jjj|_{n_k}}^{-1} \circ M^{-1}_{\pi\jjj,s_k}$. Observe that each $G_k$ is affine and $s_k\to 0$ as $k \to \infty$. Since
  \begin{equation*}
    \max_{|x|=1} |G_kx-G_k0| = \frac{s_k}{r_k}\|A_{\jjj|_{n_k}}^{-1}\|,
  \end{equation*}
  we get, by the triangle inequality and \eqref{eq:conv1}, that
  \begin{equation} \label{eq:conv3}
    1-2^{-k}  \leq \frac{s_k}{r_k}\|A_{\jjj|_{n_k}}^{-1}\| \leq 1.
  \end{equation}
for all $k\in\N$. Write $\delta = \min_{i\neq j}\dist(\varphi_i(X),\varphi_j(X))/2$ and observe that $\dist(\varphi_{\jjj|_{n_k}}(X),X\setminus\varphi_{\jjj|_{n_k}}(X))>\delta\|A_{\jjj|_{n_k}}^{-1}\|^{-1}$ for all $k\in\N$. Since $s_k \le r_k\|A_{\jjj|_{n_k}}^{-1}\|^{-1}$ and $r_k < \delta$ for all $k \in \N$ large enough, we see that $B(\pi\jjj,s_k)\cap(X\setminus\varphi_{\jjj|_{n_k}}(X))=\emptyset$ for all $k \in \N$ large enough. Thus, by reindexing the sequences, we have shown that $B(\pi\jjj,s_k)\cap X\subset\varphi_{\jjj|_{n_k}}(X)$ for all $k\in\N$.

Let $G$ be an accumulation point of the sequence $(G_k)_{k\in\N}$. To finish the proof, it is suffices to show that $G$ is linear, $\|G\|=1$, $\Ker(G)=\vartheta_1(\jjj)$, and $GF=F$. Since, by \eqref{eq:conv1},
$$
  |G0|=\lim_{k\to\infty}|G_k0|=\lim_{k\to\infty}\frac{1}{r_k}|\fii_{\jjj|_{n_k}}^{-1}(\pi\jjj)-\pi\iii|\leq\lim_{k\to\infty}2^{-k}=0,
$$
we see that $G$ is linear. Hence, $\frac{s_k}{r_k}A_{\jjj|_{n_k}}^{-1}\to G$ as $k\to\infty$, where $\frac{s_k}{r_k}A_{\jjj|_{n_k}}^{-1}$ is the linear part of $G_k$, and, consequently, \eqref{eq:conv3} implies $\|G\|=1$. This also shows that $G$ cannot be of rank zero. Observe that $G_k(B(0,1))$ is an ellipse with the lengths of the principal semiaxes $\frac{s_k}{r_k}\|A_{\jjj|_{n_k}}^{-1}\|$ and $\frac{s_k}{r_k}\|A_{\jjj|_{n_k}}\|^{-1}$. Recalling that $\nu$ has a simple Lyapunov spectrum, we see that the ratio $\|A_{\jjj|_{n_k}}^{-1}\|/\|A_{\jjj|_{n_k}}\|^{-1} \to \infty$ as $k \to \infty$. Therefore, $G$ cannot be of rank two either.
By Lemma \ref{lem:visit}, we have $\chi_1(\nu) = -\lim_{k\to\infty} \tfrac{1}{n_k}\log\|A_{\jjj|_{n_k}}^{-1}|\vartheta_1(\jjj)\|^{-1}$ and hence, by \eqref{eq:conv3},
\begin{equation} \label{eq:conv4}
\begin{split}
  \|G|\vartheta_1(\jjj)\|&=\lim_{k\to\infty}\|G_k|\vartheta_1(\jjj)\|=\lim_{k\to\infty}\frac{s_k}{r_k}\|A_{\jjj|_{n_k}}^{-1}|\vartheta_1(\jjj)\|\\
  &\leq \lim_{k\to\infty}\frac{\|A_{\jjj|_{n_k}}^{-1}|\vartheta_1(\jjj)\|}{\|A_{\jjj|_{n_k}}^{-1}\|}\leq\lim_{k\to\infty}\exp(n_k(\chi_1(\nu)-\chi_2(\nu)+2\varepsilon))=0,
\end{split}
\end{equation}
where in the last inequality, $\varepsilon>0$ is arbitrary but strictly smaller than $\frac12(\chi_2(\nu)-\chi_1(\nu))$. This shows that $\Ker(G)=\vartheta_1(\jjj)$. Finally, let $\hat F\in F$, $\hat{L}\in L$, and $\hat\vartheta_1(\jjj) \in \vartheta_1(\jjj)$ be unit vectors. Then $\hat{F}=p\hat{L}+q\hat\vartheta_1(\jjj)$ for some $p,q\in\R$. By \eqref{eq:conv4}, \eqref{eq:conv3}, \eqref{eq:conv2}, and \cite[Proposition~3.1]{BougerolLacroix1985}, we have
$$
  G(\hat{F})=pG(\hat{L})=p\lim_{k\to\infty}\frac{s_k}{r_k}A_{\jjj|_{n_k}}^{-1}\hat{L}=p\lim_{k\to\infty}\frac{\|A_{\jjj|_{n_k}}^{-1}|L\|}{\|A_{\jjj|_{n_k}}^{-1}\|}\hat{F}=p|\cos(\sphericalangle(L,\vartheta_1(\jjj)^\perp))|\hat{F}.
$$
This shows that $GF=F$ and also finishes the proof.
\end{proof}

We can now turn the information on the convergence of the magnifications of the affine mapping into information on the tangent structure of the self-affine set.

\begin{proposition} \label{thm:main-tech}
  Let $X$ be a planar self-affine set satisfying the strong separation condition and $\nu$ be a fully supported Bernoulli measure having simple Lyapunov spectrum. Then there exists a set $\Omega\subset\Sigma$ with $\nu(\Omega)=1$ such that for every $\iii,\jjj \in \Omega$, $T_\iii \in \Tan(X,\pi\iii)$, and $F\in\spt(\mu_F)$ with $\sphericalangle(F,\vartheta_1(\iii)),\sphericalangle(F,\vartheta_1(\jjj))>0$ there is $T_\jjj \in \Tan(X,\pi\jjj)$ for which
  $$
    \nproj_F^{\vartheta_1(\jjj)}(T_\jjj) \subset T_\iii\cap F \quad\text{or}\quad -\nproj_F^{\vartheta_1(\jjj)}(T_\jjj) \subset T_\iii\cap F.
  $$
  Furthermore, if $X$ also satisfies the projection condition, then there exists an at most countable set $B_\iii$ such that for every $\iii,\jjj \in \Omega$, $T_\iii \in \Tan(X,\pi\iii)$, and $F\in\spt(\mu_F) \setminus B_\iii$ with $\sphericalangle(F,\vartheta_1(\iii)),\sphericalangle(F,\vartheta_1(\jjj))>0$ there is $T_\jjj \in \Tan(X,\pi\jjj)$ for which
  $$
    \nproj_F^{\vartheta_1(\jjj)}(T_\jjj) = T_\iii\cap F \quad\text{or}\quad -\nproj_F^{\vartheta_1(\jjj)}(T_\jjj) = T_\iii\cap F.
  $$
\end{proposition}

\begin{proof}
	Fix $\iii,\jjj\in\Omega$ and let $T_\iii\in\Tan(X,\pi\iii)$. By definition, there exists a decreasing sequence $(r_k)_{k\in\N}$ of positive real numbers converging to zero so that $M_{\pi\iii,r_k}(X)\cap B(0,1)\to T_\iii$ in Hausdorff distance. Let $F\in\spt(\mu_F)$ be such that $\sphericalangle(\vartheta_1(\iii),F),\sphericalangle(\vartheta_1(\jjj),F)>0$. Then, by Proposition~\ref{prop:maintech}, there exist a sequence $(s_k)_{k\in\N}$ of positive real numbers converging to zero and an unbounded sequence $(n_k)_{k\in\N}$ of natural numbers such that
	$$
  	M_{\pi\iii,r_k}\circ\fii_{\jjj|_{n_k}}^{-1}\circ M^{-1}_{\pi\jjj,s_k}(M_{\pi\jjj,s_k}(X)\cap B(0,1)) \subset M_{\pi\iii,r_k}(X)\cap B(0,1).
	$$
	Furthermore, there exists a sequence $(k_m)_{m \in \N}$ of natural numbers such that $M_{\pi\jjj,s_{k_m}}(X)\cap B(0,1)\to T_{\jjj}\in\Tan(X,\pi\jjj)$ in Hausdorff distance and $G_{k_m}= M_{\pi\iii,r_{k_m}} \circ \fii_{\jjj|_{n_{k_m}}}^{-1} \circ M^{-1}_{\pi\jjj,s_{k_m}}\to G$, where either $G=\nproj_F^{\vartheta_1(\jjj)}$ or $G=-\nproj_F^{\vartheta_1(\jjj)}$. Thus, we get
	$$
    M_{\pi\iii,r_{k_m}} \circ\fii_{\jjj|_{n_k}}^{-1}(X\cap B(\pi\jjj,s_{k_m}))\to G(T_{\jjj}) \subset T_\iii \cap F
	$$
	in Hausdorff distance.
	
	Let us then assume that the projection condition holds. Fix $F\in\spt(\mu_F)\setminus B_\iii$, where $B_\iii$ is an at most countable set defined in Lemma~\ref{lem:omega2}. To prove the remaining inclusion, we are required to get arbitrarily close to $y \in T_\iii \cap F$ by the approximations of the projections of the tangent set $T_\jjj$. In other words, it suffices to show that for every $\varepsilon>0$ and $y\in T_\iii\cap F$ there exists $m_0\in\N$ such that for every $m\geq m_0$ it holds that $B(y,\varepsilon)\cap G_m(T_m)\neq\emptyset$, where $G_m=M_{\pi\iii,r_{k_m}} \circ\fii_{\jjj|_{n_k}}^{-1}\circ M_{\pi\jjj,s_{k_m}}^{-1}$ and $T_m=M_{\pi\jjj,s_{k_m}}(X)\cap B(0,1)$.
	Let $L\in\spt(\mu)$ be as in the proof of Proposition~\ref{prop:maintech} (especially, see \eqref{eq:conv2}), and write $\alpha_k=\sphericalangle(L, \vartheta_1(\jjj|_k))$ and $\alpha = \sphericalangle(L, \vartheta_1(\jjj)) > 0$. Note that, by \cite[Lemma 2.1]{KaenmakiKoivusaloRossi2017}, it holds that $\alpha_k \to \alpha$ as $k \to \infty$.

  Let us assume that $|y| < \sin(\alpha/2) $ and fix $0<\varepsilon<\sin(\alpha/2)-|y|$. By the definition of tangent sets, there are points of the magnification $M_{\pi\iii,r_{k_m}}(X)$ that approximate $y$, and by the projection condition, there are such points also in the line $F$. In other words, by Lemma~\ref{lem:projcond} and Theorem~\ref{thm:KKR}, we may choose $m_0'\in\N$ so large that for every $m\geq m_0'$ we have $M_{\pi\iii,r_{k_m}}(X)\cap B(0,1)\cap B(y,\varepsilon/3)\cap F\neq\emptyset$. Let $\kkk_m \in \Sigma$ be such that $M_{\pi\iii,r_{k_m}}(\pi\kkk_m)\in M_{\pi\iii,r_{k_m}}(X)\cap B(0,1)\cap B(y,\varepsilon/3)\cap F$. Choose $n\in\N$ so that $M_{\pi\iii,r_{k_m}}(\pi[\kkk_m|_n])\subset B(M_{\pi\iii,r_{k_m}}(\pi\kkk_m),\varepsilon/3)$ and, recalling that $F \notin B_\iii$, let $\kappa$ be the supremum of positive real numbers such that $M_{\pi\iii,r_{k_m}}(\pi[\kkk_m|_n])\setminus B(F,\kappa)$ contains a point on both open half-planes determined by $F$, where $B(F,\kappa)$ denotes  the $\kappa$-neighbourhood of the subspace $F$ in $\R^2$. Then there exists $m_0$ such that for every $m\geq m_0$ it holds that $G_m(T_m)\subset B(F,\kappa/2)$ and moreover, $\|G_m|L\| \geq \sin(\alpha_{k_m}) >|y|+\varepsilon$. Thus, again by Lemma~\ref{lem:projcond}, $M_{\pi\iii,r_{k_m}}(\pi[\kkk_m|_n])\cap G_m(T_m)\neq\emptyset$. Now we have that
	\begin{equation}
	 \label{eq:alphacase1}
	 G(T_\jjj) \supset T_\iii \cap B(0,\sin(\alpha/2)) \cap F.
	\end{equation}
	But since $T_\jjj$ is a rotated comb and $G$ is a projection along $\vartheta_1(\jjj)$, it holds that $G(T_\jjj \cap B(0,\sin(\alpha/2))) = G(T_\jjj) \cap B(0,\sin(\alpha/2))$. Hence we have shown that
	\begin{equation}
	 \label{eq:alphacase2}
	 G(T_\jjj \cap B(0,\sin(\alpha/2))) = T_\iii \cap B(0,\sin(\alpha/2)) \cap F.
	\end{equation}
	Finally, the claim in full generality is obtained by repeating the above proof for the tangent $T'_\iii$ for which $M_{0,\sin(\alpha/2)}(T'_\iii) \cap B(0,1) = T_\iii$.
\end{proof}

In the presence of the projection condition, the previous proposition shows that the projection of one tangent set is a slice of another. This, together with the fact that the tanget sets are combs, sharpens the previous result as follows. Recall that we have identified all the combs of the form $(\ell \times \{0\}) \cap B(0,1)$, where $\ell$ is an interval containing at least one of the intervals $[-1,0]$ and $[0,1]$.

\begin{proposition}\label{prop:rotatetangent}
  Let $X$ be a planar self-affine set satisfying the strong separation condition and the projection condition. If $\nu$ is a fully supported Bernoulli measure having simple Lyapunov spectrum, then there exists a set $\Omega\subset\Sigma$ with $\nu(\Omega)=1$ such that for every $\iii,\jjj \in \Omega$ there is a rotation $O_{\iii,\jjj}$ for which
  $$
    O_{\iii,\jjj}(T_\iii) \in \Tan(X,\pi\jjj)
  $$
  for all $T_\iii \in \Tan(X,\pi\iii)$.
\end{proposition}

\begin{proof}
  Let $\Omega$ be as in Proposition~\ref{thm:main-tech}. By Theorem~\ref{thm:KKR}, we may assume that for every $\iii\in\Omega$ each tangent set $T_\iii\in\Tan(X,\pi\iii)$ is a rotated comb. Fix $\iii,\jjj\in\Omega$ and $T_\iii\in\Tan(X,\pi\iii)$, and let $B_\iii$ be as in Lemma~\ref{lem:omega2}. Recalling Lemma \ref{lem:omega}, choose $F\in\spt(\mu_F)\setminus B_\iii$ such that $\sphericalangle(F,\vartheta_1(\iii)),\sphericalangle(F,\vartheta_1(\jjj))>0$. Let $T'_\iii$ be the tangent at $\pi\iii$ for which it holds that $ T_\iii = M_{0,\eta} (T'_\iii) \cap B(0,1)$, where $\eta = |\sin(\sphericalangle(F,\vartheta_1(\jjj)))|$. By Proposition~\ref{thm:main-tech},  there exists $T_\jjj\in\Tan(X,\pi\jjj)$ such that
  \begin{equation}\label{eq:that}
    \nproj_F^{\vartheta_1(\jjj)}(T_\jjj)=T'_\iii\cap F \quad\text{or}\quad -\nproj_F^{\vartheta_1(\jjj)}(T_\jjj)=T'_\iii\cap F.
  \end{equation}
  Let $R$ be the rotation such that $R\vartheta_1(\jjj)=\vartheta_1(\iii)$, and let $P$ be the reflection through $\vartheta_1(\iii)$. Then let $O=R$ or $O=PR$ depending on whether the equation \eqref{eq:that} is realized by $\nproj_F^{\vartheta_1(\jjj)}$ or $-\nproj_F^{\vartheta_1(\jjj)}$. It is easy to see that $\nproj_F^{\vartheta_1(\iii)}\circ O^{-1} = \nproj_{F}^{O \vartheta_1(\iii)}$. Since $T'_\iii$ is a rotated comb,
  \[
   M_{0,\eta}(T'_{\iii}) \cap B(0,1)
   =
   (\nproj_F^{\vartheta_1(\iii)})^{-1}(M_{0,\eta}T_{\iii}\cap F) \cap B(0,1) .
  \]
  Hence, by \eqref{eq:that},
  \begin{align*}
   T_\iii
   &=
   M_{0,\eta} ( T'_\iii) \cap B(0,1) \\
   &=
   (\nproj_F^{\vartheta_1(\iii)})^{-1}(M_{0,\eta}(T'_{\iii})\cap F) \cap B(0,1) \\
   &=
   (\nproj_F^{\vartheta_1(\iii)})^{-1}( \nproj_F^{\vartheta_1(\jjj)}M_{0,\eta}(T'_\jjj) ) \cap B(0,1) \\
   &=
   O \circ M_{0,\eta} \circ (\nproj_F^{\vartheta_1(\iii)} \circ O )^{-1}( \nproj_F^{\vartheta_1(\jjj)}(T_\jjj) ) \cap B(0,1) \\
   &=
   O \circ M_{0,\eta} \circ (\nproj_F^{\vartheta_1(\jjj)} )^{-1}( \nproj_F^{\vartheta_1(\jjj)}(T_\jjj) ) \cap B(0,1) \\
   &=
   O \circ M_{0,\eta} (T_\jjj) \cap B(0,1) \\
   &=
   O (T'_\jjj),
  \end{align*}
  where $T'_\jjj = M_{0,\eta} (T_\jjj) \cap B(0,1) \in \Tan(X,\pi\jjj)$.
\end{proof}

Now we are ready to prove our first main theorem.

\begin{proof}[Proof of Theorem~\ref{thm:maininformal}]
  Let $\Omega \subset \Sigma$ be as in Proposition \ref{prop:rotatetangent}. Our task is to show that for every $\iii,\jjj\in\Omega$ we have
  $$
    \{O_{\iii}T:T\in\Tan(X,\pi\iii)\}=\{O_\jjj T:T\in\Tan(X,\pi\jjj)\},
  $$
  where $O_{\iii}$ and $O_\jjj$ are the rotations that take $\vartheta_1(\iii)$ and $\vartheta_1(\jjj)$ to the $x$-axis, respectively. By symmetry, it suffices to prove that
  $$
    \{O_{\iii,\jjj} T:T\in\Tan(X,\pi\iii)\} \subset \Tan(X,\pi\jjj),
  $$
  where $O_{\iii,\jjj}$ is the rotation that takes $\vartheta_1(\iii)$ to $\vartheta_1(\jjj)$. But this follows immediately from Proposition~\ref{prop:rotatetangent}.
\end{proof}

\section{Assouad dimension} \label{sec:assouad-dimension}

In this section, we study the Assouad dimension of planar self-affine sets. To that end, we require a slightly modified version of Proposition \ref{thm:main-tech}: instead of comparing two tangent sets, we study the relation between tangent sets and the self-affine set. For every $\iii\in\Sigma$ let $B_\iii$ be as in Lemma~\ref{lem:omega2}, and for $F\in B_\iii$ let
\begin{equation} \label{eq:vfi}
\begin{split}
  V_{F,\pi\iii}=\{\pi\kkk-\pi\iii \in (X-\pi\iii) \cap F : \;&\text{there exist }n\in\N\text{ such that } \pi[\kkk|_n] \text{ is contained in} \\
	&\text{one of the closed half-planes determined by } F+\pi\iii \}.
\end{split}
\end{equation}
We extend the definition of $V_{F,\pi\iii}$ for all $F\in\RP$ by setting $V_{F,\pi\iii}=\emptyset$ for $F \in \RP\setminus B_\iii$. We write $\delta=\min_{i\neq j}\mathrm{dist}(\varphi_i(X),\varphi_j(X))/2$ and for a fixed fully supported Bernoulli measure $\nu$ having simple Lyapunov spectrum we set $\beta_\iii=\sup_{L\in\Theta_\iii}\sphericalangle(\vartheta_1(\iii),L)$ for all $\iii \in \Omega$ and $\beta=\sup_{\iii\in\Omega}\beta_\iii=\sup_{(\iii,L)\in\Theta}\sphericalangle(\vartheta_1(\iii),L)$, where $\Theta \subset \Sigma \times \RP$ is as in Lemma \ref{lem:visit}, $\Omega \subset \Sigma$ as in Lemma~\ref{lem:omega}, and $\Theta_\iii \subset \RP$ as in the proof of Lemma \ref{lem:omega}.

\begin{proposition}\label{prop:balls}
  Let $X$ be a planar self-affine set satisfying the strong separation condition and $\nu$ be a fully supported Bernoulli measure having simple Lyapunov spectrum. Then there exists a set $\Omega \subset \Sigma$ with $\nu(\Omega)=1$ such that for every $x\in X$, $\jjj \in \Omega$, $F\in\spt(\mu_F)$ with $\sphericalangle(F,\vartheta_1(\jjj))>0$ there is $T_\jjj\in\Tan(X,\pi\jjj)$ for which
  \begin{equation*} 
    M^{-1}_{x,\delta}\circ\nproj_F^{\vartheta_1(\jjj)}(T_\jjj) \subset (X-x) \cap F \cap B(0,\delta)
  \end{equation*}
  or
  \begin{equation*} 
    -M^{-1}_{x,\delta}\circ\nproj_F^{\vartheta_1(\jjj)}(T_\jjj) \subset (X-x) \cap F \cap B(0,\delta).
  \end{equation*}
   Furthermore, if $X$ also satisfies the projection condition, then
  \begin{equation*} 
    M^{-1}_{x,\delta\sin(\beta_\jjj/4)}\circ\nproj_F^{\vartheta_1(\jjj)}(T_\jjj) \supset (X-x) \cap F \cap B(0,\delta \sin(\beta_\jjj/4)) \setminus V_{F,x}
  \end{equation*}
  or
  \begin{equation*} 
    -M^{-1}_{x,\delta\sin(\beta_\jjj/4)}\circ\nproj_F^{\vartheta_1(\jjj)}(T_\jjj) \supset (X-x) \cap F \cap B(0,\delta \sin(\beta_\jjj/4)) \setminus V_{F,x},
  \end{equation*}
  where $V_{F,x}$ is as in \eqref{eq:vfi}.
\end{proposition}

\begin{proof}
	Let $\Omega$ be as in Lemma~\ref{lem:omega}. Fix $x\in X$ and $\jjj\in\Omega$ and let $(r_k)_{k \in \N}$ be a sequence of real numbers satisfying $r_k = \delta$ for all $k \in \N$. By assumption, we have $F\in\spt(\mu_F)$ so that $\sphericalangle(F,\vartheta_1(\jjj))>0$. Then by Proposition~\ref{prop:maintech}, there exist a sequence $(s_k)_{k\in\N}$ of positive real numbers converging to zero and an unbounded sequence $(n_k)_{k\in\N}$ of natural numbers such that
	$$
    \fii_{\jjj|_{n_k}}^{-1} \circ M^{-1}_{\pi\jjj,s_k}(M_{\pi\jjj,s_k}(X)\cap B(0,1)) \subset X \cap B(x,\delta).
	$$
  Furthermore, there exists a sequence $(k_m)_{m \in \N}$ of natural numbers such that $M_{\pi\jjj,s_{k_m}}(X) \cap B(0,1) \to T_{\jjj} \in \Tan(X,\pi\jjj)$ in Hausdorff distance and
  $$
    G_{k_m} := M_{x,\delta} \circ \fii_{\jjj|_{n_{k_m}}}^{-1} \circ M^{-1}_{\pi\jjj,s_{k_m}} \to G \in \{\nproj_F^{\vartheta_1(\jjj)},-\nproj_F^{\vartheta_1(\jjj)}\}
  $$
  as $m\to\infty$. Hence, $\fii_{\jjj|_{n_k}}^{-1}(X\cap B(\pi\jjj,s_{k_m})) \to M^{-1}_{x,\delta} \circ G(T_{\jjj})$ in Hausdorff distance. By compactness of $X$, we conclude $M^{-1}_{x,\delta} \circ G(T_{\jjj}) \subset X \cap B(x,\delta)$.
	
  Let us then assume that the projection condition holds. It suffices to prove that if $\pi\kkk\in X\cap(F+x)\cap B(x,\delta \sin(\beta_\jjj/4)) \setminus (V_{F,x}+x)$, then $\pi\kkk\in G(T_{\jjj})$. By compactness of $G(T_{\jjj})$, it is enough to show that for every $\varepsilon>0$ there exists $m_0\in\N$ such that for every $m\geq m_0$ it holds that $G_{k_m}(T_{k_m})\cap B(\pi\kkk,\varepsilon)\neq\emptyset$, where $T_{k_m}=M_{\pi\jjj,s_{k_m}}(X)\cap B(0,1)$. Let $n\in\N$ be large enough so that $\pi[\kkk|_n]\subset B(\pi\kkk,\varepsilon)$. Then $F+x$ divides $\pi[\kkk|_n]$ into two parts, and let $\kappa>0$ be the supremum of positive real numbers such that $\pi[\kkk|_n]\setminus B(F+x,\kappa)$ contains a point on both open half-planes determined by $F+x$. Choosing $L\in\Theta_\jjj$ such that $\sphericalangle(\vartheta_1(\jjj),L)>\beta_\jjj/2$, we get
  $\|G|L\|=|\sin(\sphericalangle(\vartheta_1(\jjj),L))|\|G\|\geq\delta\sin(\beta_\jjj/2)$.

  Since $G_{k_m}(T_{k_m})\to G(T)\subset F+x$ in Hausdorff distance, there exists $m_0$ such that for every $m\geq m_0$ it holds that $G_{k_m}(T_{k_m})\subset B(F+x,\kappa/2)$. Furthermore, since $\|G_{k_m}|L\|\to \|G|L\| \geq \delta \sin(\beta/2)$ by construction, we may choose $m_0$ so large that $|x-\pi\kkk| \leq \sin(\beta_\jjj/4) \delta < \|G_{k_m}|L\|$. Hence, by Lemma~\ref{lem:projcond}, $G_{k_m}(T_{k_m})\cap\pi[\kkk|_n]\neq\emptyset$ and thus $G_{k_m}(T_{k_m})\cap B(\pi\kkk,\varepsilon)\neq\emptyset$ for every $m\geq m_0$.
\end{proof}

The following theorem yields the upper bound in Theorem \ref{thm:assouad}. By Proposition \ref{thm:KOR}, to bound the Assouad dimension, it is enough to bound the Hausdorff dimension of tangent sets. This can be done by the previous proposition as it shows that a slice of a self-affine set contains a projection of a tangent set. We remark that the assumption on the existence of $n$-step Bernoulli measure having close to maximal dimension is crucial in the proof. Note also that the projection condition is not needed in the upper bound.

\begin{theorem}\label{thm:assouadv2}
  Let $X$ be a planar self-affine set satisfying the strong separation condition. If for every $\eps>0$ there exists a fully supported $n$-step Bernoulli measure $\nu$ having simple Lyapunov spectrum such that $\dim(\pi\nu) \ge \dimh(X)-\eps$, then
  $$
		\dima(X) \leq 1+\sup_{\atop{x\in X}{F\in X_F}} \dimh(X \cap (F+x) \cap B(x,\delta')).
  $$
  for all $0<\delta'<\min_{i\neq j}\mathrm{dist}(\varphi_i(X),\varphi_j(X))$.
\end{theorem}

\begin{proof}
  By Proposition~\ref{thm:KOR}, $\dima(X)=\max\{\dimh(T) : T\in\Tan(X)\}$. It is therefore enough to estimate the Hausdorff dimension of all weak tangent sets from above. Fix $0<\delta'<\min_{i\neq j}\mathrm{dist}(\varphi_i(X),\varphi_j(X))$. Let us first show that for every weak tangent set $T \in \Tan(X)$ there exists a linear map $G$ and a point $y\in X$ such that
	\begin{enumerate}
		\item $\delta'\min\{\|A_i^{-1}\|^{-1} : i \in \{1,\ldots,N\}\} \le \|G\| \le \delta'$,
		\item $G(T)+y \subset X\cap B(y,\delta')$.
	\end{enumerate}
  Indeed, let $(\iii_k)_{k\in\N}$ be a sequence of infinite words in $\Sigma$ and $(s_k)_{k\in\N}$ be a sequence of positive real numbers converging to zero such that $M_{\pi\iii_k,s_k}(X)\cap B(0,1)\to T$ in Hausdorff distance. Write $n_k:=\max\{n\in\N : \varphi^{-1}_{\iii_k|_n}(B(\pi\iii_k,s_k)) \subset B(\pi\sigma^n\iii_k,\delta')\}$ and observe that then
	$$
    \delta'\min\{\|A_i^{-1}\|^{-1} : i \in \{1,\ldots,N\}\} < \|A_{\iii_k|_{n_k}}^{-1}\|s_k \leq \delta'.
	$$
	Since $0<\delta'<\dist(\varphi_i(X),\varphi_j(X))$ whenever $i \neq j$, we see that $X \cap B(\pi\iii_k,s_k)\subset\varphi_{\iii_k|_{n_k}}(X)$. By compactness, without loss of generality, we may assume that there exists $y\in X$ such that $\pi\sigma^{n_{k}}\iii_{k} \to y$. Hence, $\varphi_{\iii_k|_{n_{k}}}^{-1} \circ M_{\pi\iii_k,s_{k}}^{-1} \to G+y$ uniformly on $B(0,1)$, and thus, $\varphi_{\iii_k|_{n_{k}}}^{-1}\circ M_{\pi\iii_k,s_{k}}^{-1}(X \cap B(\pi\iii_k,s_k)) \to G(T)+y$ in Hausdorff distance. Since $\varphi_{\iii_k|_{n_k}}^{-1}(X \cap B(\pi\iii_k,s_k))\subset X$, we get, by compactness, that $G(T)+y\subset X\cap B(y,\delta')$. We have therefore proved the claim.

  Fix $T \in \Tan(X)$ and let $G$ and $y$ be as above. Since $\|G\| \ge \delta'\min\{\|A_i^{-1}\|^{-1} : i \in \{1,\ldots,N\}\} > 0$, we see that $G$ is of rank one or rank two. If $G$ is of rank one, then, by definition, $\mathrm{Im}(G) \in X_F$ and $G(T)+y \subset X \cap B(y,\delta') \cap (\mathrm{Im}(G)+y)$. It follows that
  $$
    \dimh(T) \leq 1+\dimh(G(T)+y) \leq 1+\sup_{\atop{x\in X}{F\in X_F}}\dimh(X \cap(F+x) \cap B(x,\delta')).
  $$
  If $G$ is of rank two, then
  $$
    \dimh(T)=\dimh(G(T)+y)\leq\dimh(X).
  $$
  Observe that if $\dimh(X)\leq1$, then the claim of the theorem holds trivially. Without loss of generality, we may assume that $\dimh(X)>1$.

  Fix $\varepsilon>0$ and let $\nu$ be an $n$-step Bernoulli measure having simple Lyapunov spectrum such that $\dim(\pi\nu) \ge \dimh(X)-\eps$. By \cite[Propositions~5.8 and 5.9]{BaranyKaenmaki2017}, for $\mu_F$-almost every $F$ and $\nu$-almost every $\iii$, we have
  $$
    \dim(\pi\nu) = \dim(\proj_{F^\perp}\pi\nu)+\dim((\pi\nu)^F_{\pi\iii}) \leq 1+\sup_{\atop{x\in X}{F\in\spt(\mu_F)}}\dimh(X\cap(F+x)\cap B(x,\delta')),
  $$
  where $\{(\pi\nu)^F_{\pi\iii}\}$ is the family of conditional measures of $\pi\nu$ such that $\spt((\pi\nu)^F_{\pi\iii})\subset(F+\pi\iii)\cap X$ and $\pi\nu=\int_\Sigma(\pi\nu)^F_{\pi\jjj} \dd\nu(\jjj)$. Recalling that $\spt(\mu_F)\subset X_F$, we get
  $$
    \dimh(T) \leq \dimh(X) \leq 1+\sup_{\atop{x\in X}{F\in X_F}} \dimh(X\cap(F+x)\cap B(x,\delta')) + \eps.
  $$
  Letting $\eps \downarrow 0$ finishes the proof.
\end{proof}

Let us now turn our attention to the lower bound. We begin with a small technical lemma.

\begin{lemma}\label{lem:division}
  Let $A \subset \R$ be compact. Then for every $\eps>0$ there exists a point $x \in A$ such that
  \begin{equation*}
    \min\{\dimh(A\cap(-\infty,x]),\dimh(A\cap[x,\infty))\} \geq \dimh(A)-\varepsilon.
  \end{equation*}
\end{lemma}

\begin{proof}
  Suppose for a contradiction that this is not the case. In other words, we assume that there is $\eta>0$ such that
  $$
    \min\{\dimh(A\cap(-\infty,x]),\dimh(A\cap[x,\infty))\} \leq \dimh(A)-\eta
  $$
  for all $x\in A$. Write $y_1 = \sup\{x\in A : \dimh(A\cap(-\infty,x]) \leq \dimh(A)-\eta\}$ and $y_2=\inf\{x\in A : \dimh(A\cap[x,\infty)) \leq \dimh(A)-\eta\}$. Since $A$ is compact, we have $y_1,y_2 \in A$. If either $y_1\geq y_2$ or $y_1<y_2$ and $A\cap(y_1,y_2)=\emptyset$, then there exist an increasing sequence $(x_n)_{n\in\N}$ of real numbers converging to $y_1$ and a decreasing sequence $(z_n)_{n\in\N}$ of real numbers converging to $y_2$ so that
  $$
    A \setminus \{y_1,y_2\} \subset \biggl(\bigcup_{n \in \N} A \cap (-\infty,x_n]\biggr) \cup \biggl(\bigcup_{n \in \N} A \cap [z_n,\infty)\biggr),
  $$
  where $\max\{ \dimh(A \cap (-\infty,x_n]), \dimh(A \cap [z_n,\infty)) \} \le \dimh(A)-\eta$ for all $n \in \N$. Since the Hausdorff dimension is countably stable, this is a contradiction. Therefore, $y_1<y_2$ and $A \cap (y_1,y_2) \neq \emptyset$. But this cannot be the case either, since, by choosing $x\in A \cap (y_1,y_2)$, we have $\min\{\dimh(A \cap (-\infty,x]),\dimh(A \cap [x,\infty))\} > \dimh(A)-\eta$ which is again a contradiction.
\end{proof}

The following theorem yields the lower bound in Theorem \ref{thm:assouad}. In the proof, we use similar approach as in the proof of Theorem \ref{thm:assouadv2} but this time we assume the projection condition and apply the relevant inclusions from Proposition \ref{prop:balls}.

\begin{theorem}\label{thm:assouadv3}
  Let $X$ be a planar self-affine set satisfying the strong separation condition and the projection condition. If $\nu$ is a fully supported Bernoulli measure having simple Lyapunov spectrum, then
  \begin{equation*}
    \dima(X) \ge 1 + \sup_{\atop{x \in X}{F \in \spt(\mu_F)}} \dimh(X \cap (F+x) \cap B(x,\delta \sin(\beta/4) ) ),
  \end{equation*}
  where $\mu_F$ is the Furstenberg measure with respect to $\nu$.
\end{theorem}

\begin{proof}
  Fix $\varepsilon>0$ and let $x\in X$ and $F\in\spt(\mu_F)$ be such that
  \begin{equation} \label{eq:assouadv3-1}
    \dimh(X\cap(F+x)\cap B(x,\delta \sin(\beta/4) )) \ge \sup_{\atop{x'\in X}{F'\in\spt(\mu_F)}}\dimh(X\cap(F'+x')\cap B(x',\delta \sin(\beta/4) ))-\varepsilon.
  \end{equation}
  Let $V_{F,x}$ be as in \eqref{eq:vfi}. Suppose first that
  \begin{equation} \label{eq:assouadv3-case1}
    \dimh(X\cap(F+x)\cap B(x,\delta \sin(\beta/4) ))=\dimh(X\cap(F+x)\cap B(x,\delta \sin(\beta/4) )\setminus (V_{F,x}+x)).
  \end{equation}
  Let $\Omega \subset \Sigma$ be as in Proposition \ref{prop:balls}. Since, by \eqref{eq:furst} and Lemma~\ref{lem:visit}, $\lim_{n\to\infty}\frac{1}{n}\log\|A_{\jjj|_n}^{-1}|F\|=\chi_2(\nu)$ and $\lim_{n\to\infty}\frac{1}{n}\log\|A_{\jjj|_n}^{-1}|\vartheta_1(\jjj)\|=\chi_1(\nu)$ for $\nu$-almost all $\jjj \in \Sigma$, there exists $\jjj\in\Omega$ such that $\sphericalangle(\vartheta_1(\jjj),F) > 0$ and $\beta_\jjj>\beta/2$. Therefore, by Proposition~\ref{prop:balls}, there exist $T_\jjj\in\Tan(X,\pi\jjj)$ and a linear map $G \in \{\nproj_F^{\vartheta_1(\jjj)}, -\nproj_F^{\vartheta_1(\jjj)}\}$ such that
  $$
    G(T_\jjj)+x\supset X\cap(F+x)\cap B(x,\delta \sin(\beta/4) )\setminus (V_{F,x}+x).
  $$
  By Theorem~\ref{thm:KKR}, $O_\jjj T_\jjj$ is a comb, where $O_\jjj$ is the rotation that takes $\vartheta_1(\jjj)$ to the $x$-axis, and thus,
  \begin{equation} \label{eq:assouadv3-2}
    \dimh(T_\jjj) = 1+\dimh(G(T_\jjj)) \geq 1+\dimh(X\cap(F+x)\cap B(x,\delta \sin(\beta/4) )).
  \end{equation}
  Furthermore, by Proposition~\ref{thm:KOR}, \eqref{eq:assouadv3-2}, and \eqref{eq:assouadv3-1},
	\[
  \begin{split}
    \dima(X) &= \max\{\dimh(T') : T'\in\Tan(X)\} \geq \dimh(T_\jjj) \\
    &\geq 1+\sup_{\atop{x'\in X}{F'\in\spt(\mu_F)}}\dimh(X\cap(F'+x')\cap B(x',\delta \sin(\beta/4) ))-\varepsilon.
	\end{split}
	\]
  By letting $\eps \downarrow 0$, the proof follows under the assumption \eqref{eq:assouadv3-case1}.

  Let us then assume that
  \begin{equation} \label{eq:assouadv3-case2}
    \dimh(X\cap(F+x)\cap B(x, \delta\sin(\beta/4) )) = \dimh((V_{F,x}+x) \cap B(x,\delta\sin(\beta/4)))  > 0.
  \end{equation}
  For $\pi\kkk\in(V_{F,x}+x) \cap B(x,\delta\sin(\beta/4)) $, let $n(\kkk)$ be the smallest integer such that $\pi[\kkk|_{n(\kkk)}]$ is contained in one of the closed half-planes determined by $F+x$ and $\diam\pi[\kkk|_{n(\kkk)}]<\delta\sin(\beta/4)$. Observe that there is a countable set $\{\pi\kkk_1,\pi\kkk_2,\ldots\} \subset(V_{F,x}+x) \cap B(x,\delta\sin(\beta/4))$ so that $(V_{F,x}+x) \cap B(x,\delta\sin(\beta/4))\subset \bigcup_{l\in\N} \pi[\kkk_l|_{n(\kkk_l)}]$. Relying on the strong separation condition, we may assume that the sets $\pi[\kkk_l|_{n(\kkk_l)}]$ are pairwise disjoint. Hence, for every $\varepsilon>0$ there exists $l\in\N$ such that $\dimh((F+x)\cap\pi[\kkk_l|_{n(\kkk_l)}]) \ge \dimh((V_{F,x}+x) \cap B(x,\delta\sin(\beta/4))) -\varepsilon$. For a point $z \in F+x$, denote the two half-lines of $F+x$ separated by $z$ by $F^+_z+x$ and $F_z^-+x$. By Lemma~\ref{lem:division}, there exists $\pi\kkk\in (F+x)\cap\pi[\kkk_l|_{n(\kkk_l)}]$ such that
  \begin{equation} \label{eq:something-important}
  \begin{split}
    \dimh((F_{\pi\kkk}^{\pm}+x)\cap\pi[\kkk_l|_{n(\kkk_l)}]) &\ge \dimh((F+x)\cap\pi[\kkk_l|_{n(\kkk_l)}])-\varepsilon \\
    &\ge \dimh(V_{F,x}+x) \cap B(x,\delta\sin(\beta/4)) - 2\varepsilon.
  \end{split}
  \end{equation}
  Let $\nu$ be a fully supported Bernoulli measure having simple Lyapunov spectrum. Let $\Omega \subset \Sigma$ with $\nu(\Omega)=1$ be as in Lemma~\ref{lem:omega} and fix $\jjj\in\Omega$  and $L\in\Theta_\jjj$ such that $\sphericalangle(\vartheta_1(\jjj),F)>0$ and $\sphericalangle(\vartheta_1(\jjj),L)>\beta/2$. Then, by Proposition~\ref{prop:maintech}, there exist a sequence $(s_k)_{k\in\N}$ of positive real numbers converging to zero and an unbounded sequence $(n_k)_{k\in\N}$ of natural numbers such that
	$$
    \fii_{\jjj|_{n_k}}^{-1}(X\cap B(\pi\jjj,s_k))\subset X\cap B(\pi\kkk, \sin(\beta/4)^{-1} \diam(\pi[\kkk_\ell|_{n(\kkk_\ell)}])).
	$$
 To simplify notation, write $\delta'=\sin(\beta/4)^{-1}\diam(\pi[\kkk_\ell|_{n(\kkk_\ell)}])$ and $\delta''=\diam(\pi[\kkk_\ell|_{n(\kkk_\ell)}])$. There exists a sequence $(k_m)_{m\in\N}$ of natural numbers and $T_\jjj \in \Tan(X,\pi\jjj)$ such that $\pi\sigma^{n_{k_m}}\jjj \to \pi\kkk$ and
  \begin{align*}
    T_{k_m} &:= M_{\pi\jjj,s_{k_m}}(X)\cap B(0,1) \to T_\jjj, \\
    G_{k_m} &:= \fii_{\jjj|_{n_{k_m}}}^{-1} \circ M^{-1}_{\pi\jjj,s_{k_m}}\to G \in \{M_{\pi\kkk,\delta'}^{-1}\circ\nproj_F^{\vartheta_1(\jjj)},-M_{\pi\kkk,\delta'}^{-1}\circ\nproj_F^{\vartheta_1(\jjj)}\}
  \end{align*}
  in Hausdorff distance as $m \to \infty$. Hence, $\fii_{\jjj|_{n_k}}^{-1}(X\cap B(\pi\jjj,s_{k_m}))\to G(T_{\jjj})$ in Hausdorff distance.

  Let us next show that
  \begin{equation}\label{eq:need-1}
    G(T_\jjj)\supset(F_{\pi\kkk}^++x)\cap\pi[\kkk_l|_{n(\kkk_l)}] \quad\text{or}\quad G(T_\jjj)\supset(F_{\pi\kkk}^-+x)\cap\pi[\kkk_l|_{n(\kkk_l)}].
  \end{equation}
  The point $\pi\sigma^{n_{k_m}}\jjj$ divides the line segment $G_{k_m}(L \cap B(0,1))$ into two line segments. Denote them by $L_m^+$ and $L_m^-$.
  Let $P$ be the closed half-plane determined by $F+x$ which contains $\pi[\kkk_l|_{n(\kkk_l)}]$. Since $\pi\sigma^{n_{k_m}}\jjj\in \pi[\kkk_l|_{n(\kkk_l)}]$ for every $m$ large enough, we see that at least one of the line segments $L_{m}^\pm$ must be contained in $P$. Denote this segment by $L_m$. By going into a subsequence if necessary, we have that either
  $$
    L_{m} \cap B(\pi\kkk,\delta') \to(F_{\pi\kkk}^++x)\cap B(\pi\kkk,\delta'') \quad\text{or}\quad L_{m} \cap B(\pi\kkk,\delta') \to(F_{\pi\kkk}^-+x)\cap B(\pi\kkk,\delta'')
  $$
  in Hausdorff distance. We may assume, without loss of generality, that the first case is true. By compactness of $G(T_\jjj)$, it suffices to show that for every $\eta>0$ and $\pi\hhh\in (F_{\pi\kkk}^++x)\cap\pi[\kkk_l|_{n(\kkk_l)}]$ there exist $m_0\in\N$ such that $G_{k_m}(T_{k_m})\cap B(\pi\hhh,\eta)\neq\emptyset$ for all $m\geq m_0$. We may of course assume that $\pi\hhh\neq \pi\kkk$ since $\pi\kkk \in G(T_\jjj)$ trivially by the construction. Since $\pi\hhh\in(F^+_{\pi\kkk}+x)\cap\pi[\kkk_l|_{n(\kkk_l)}]$, we can choose $m_0$ so large that $|\pi\kkk-\pi\hhh| < \delta''$. Let $n\in\N$ be large enough so that $\pi[\hhh|_n]\subset B(\pi\hhh,\eta)$. Since $\pi\hhh$ is an element of the set $\pi[\kkk_l|_{n(\kkk_l)}]$ which is contained in the closed half-plane $P$, the set $\pi[\hhh|_n]$ is contained in $P$ as well. Let $\kappa$ be the supremum of positive real numbers for which $\pi[\hhh|_n]\setminus B(F+x,\kappa) \ne \emptyset$. Since $G_{k_m}(T_{k_m})\to G(T_\jjj)\subset F+x$ in Hausdorff distance, there exists $m_0$ such that $L_m \subset P \cap B(F+x,\kappa/2)$ for all $m\geq m_0$. But since we also have $\pi[\hhh|_n] \subset P$, $\pi\sigma^{n_{k_m}}\jjj \to \pi\kkk \neq \pi\hhh$ with $\pi\sigma^{n_{k_m}}\jjj$ being the other endpoint of of the line segment $L_m$, and $\diam (L_m) \geq \sin(\beta/4) \| G_m(T_m)) \| \geq \diam([\kkk_l|_{n(k_l)}])$, we have $L_m \cap \pi[\hhh|_n]\neq\emptyset$ by Lemma~\ref{lem:projcond}. See Figure~\ref{fig:projection_condition_in_use} for illustration.

\begin{figure}[t]
 \centering
\begin{tikzpicture}
\clip (-4,-3.5) rectangle (8,3); 
\begin{scope}[rotate=20] 
%
\draw[thick] (-6,0) -- (7,0); 
\node[right] at (7,0) (F) {$F+x$};
\draw[|-|,thick] (0.3,-0.2) -- (7,-0.5); 
\node at (2,-0.6) (L) {$L_m$}; 
\node at (-2,-2) (P) {$P$}; 
%
\node[above] at (0,0) (k) {$\pi \kkk$}; 
\draw[thick] (0,-0.07) -- (0,0.07); 
\node[above] at (4,0) (k) {$\pi \hhh$}; 
\draw[thick] (4,-0.07) -- (4,0.07); 
\draw[decorate, decoration={snake, segment length=5, amplitude=1}] (-1,0) -- (-1,-3) -- (5,-3) -- (5,0); 
\node at (0,-2) (km) {$\pi [\kkk_l|_{n(k_l)}]$}; 
\draw[decorate, decoration={snake, segment length=5, amplitude=1}] (3.8,0)-- (3.8,-1) -- (4.2,-1) -- (4.2,0); 
\draw (4,0) circle (1.1); 
\coordinate (c) at ($ (4,0) + (120:1.1) $);
\node[left] at (c) (bketa) {$B(\pi\hhh,\eta)$}; 
%
\draw[thick,dashed] (-6,-1) -- (7,-1); 
\draw[<->] (-3,-0.9) -- (-3,-0.1); 
\node[right] at (-3,-0.5) (kappa) {$\kappa$}; 
%
\draw (0,0) circle (6.5); 
\filldraw[white] (-45:4.8) circle (0.7); 
\draw[] (0,0) -- ++ (-45:6.5); 
\draw[decorate,decoration={text along path,text={~ ${ \diam (\pi[\kkk_l|_{n(k_l)}]) }$~}}] (0,0.2) ++ (-45:3) -- ++(-45:3.5); 
\end{scope}
\end{tikzpicture}
 \caption{By Lemma~\ref{lem:projcond}, the projection of $\pi[\hhh|_n]$ along the direction of the line segment $L_m$ is an interval. Since $L_m$ divides the cylinder $\pi[\hhh|_n]$ into two  parts, $L_m \cap \pi[\hhh|_n]$ has to be non-empty. }
 \label{fig:projection_condition_in_use}
\end{figure}
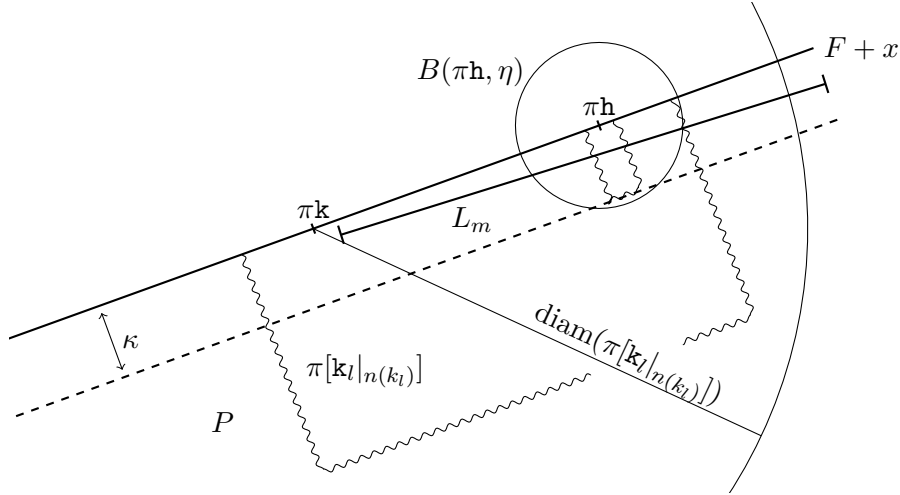

  By Theorem~\ref{thm:KKR}, $O_\jjj T_\jjj$ is a comb, where $O_\jjj$ is the rotation that takes $\vartheta_1(\jjj)$ to the $x$-axis, and thus,
  \begin{equation} \label{eq:assouadv3-4}
    \dimh(T_\jjj) = 1+\dimh(G(T_\jjj)).
  \end{equation}
  Furthermore, by Proposition~\ref{thm:KOR}, \eqref{eq:assouadv3-4}, \eqref{eq:need-1}, \eqref{eq:something-important}, \eqref{eq:assouadv3-case2}, and \eqref{eq:assouadv3-1},
  \begin{align*}
    \dima(X) &= \max\{\dimh(T'):T'\in\Tan(X)\} \ge \dimh(T_\jjj) \\
    &\ge 1 + \dimh((F^{\pm}_{\pi\kkk}+x) \cap \pi[\kkk_l|_{n(\kkk_l)}]) \\
    &\ge 1 + \dimh(X \cap (F+x) \cap B(x,\delta\sin(\beta/4))) - 2\eps \\
    &\ge 1 + \sup_{\atop{x'\in X}{F'\in\spt(\mu_F)}}\dimh(X\cap(F'+x')\cap B(x',\delta\sin(\beta/4))) - 3\eps.
  \end{align*}
  Therefore, by letting $\eps \downarrow 0$, the proof follows also under the assumption \eqref{eq:assouadv3-case2}.
\end{proof}

Relying on the estimates verified above, we are able to prove a refined version of Proposition \ref{thm:KOR} for the class of self-affine sets we are considering.

\begin{proposition}\label{thm:conf_assouad_minimality}
  Let $X$ be a planar self-affine set satisfying the strong separation condition and the projection condition. If for every $\eps>0$ there exists a fully supported $n$-step Bernoulli measure $\nu$ having simple Lyapunov spectrum such that $\dim(\pi\nu) \ge \dimh(X)-\eps$, then
  \begin{equation*}
    \dima(X) = \sup\{\dimh(T) : x \in X \text{ and } T \in \Tan(X,x) \text{ is a rotated comb} \}.
  \end{equation*}
\end{proposition}

\begin{proof}
  Fix $\eps>0$ and let $\nu$ be a fully supported $n$-step Bernoulli measure $\nu$ having simple Lyapunov spectrum such that $\dim(\pi\nu) \ge \dimh(X)-\eps$. By Lemma \ref{lem:xfissupp}, there exist two fully supported $2n$-step Bernoulli measures $\nu_1$ and $\nu_2$ having simple Lyapunov spectrum such that $\spt(\mu_F^1) \cup \spt(\mu_F^2) = X_F$, where $\mu_F^i$ is the Furstenberg measure with respect to $\nu_i$. By considering $2n$:th iterates of the maps and matrices, we may, to simplify notation, assume that $\nu,\nu_1$, and $\nu_2$ are Bernoulli measures.

  Recalling the proof of Theorem \ref{thm:assouadv3} and Theorem \ref{thm:KKR}, we see that for both $i \in \{1,2\}$ there exist $x_i \in X$ and a rotated comb $T_i \in \Tan(X,x_i)$ such that
  \[
    \dimh(T_i) \geq  1+\sup_{\atop{x\in X}{F\in\spt(\mu^i_F)}}\dimh(X\cap(F+x)\cap B(x,\delta\sin(\beta/4)))-\varepsilon
  \]
  and hence,
  \[
    \max_{i \in \{1,2\}} \dimh(T_i) \geq 1 + \sup_{\atop{x\in X}{F\in X_F }}\dimh(X\cap(F+x)\cap B(x,\delta\sin(\beta/4))) - \eps \geq \dima(X) -\eps,
  \]
 where the last inequality holds by Theorem \ref{thm:assouadv2}.
\end{proof}

Finally, we are ready to prove our second main theorem.

\begin{proof}[Proof of Theorem~\ref{thm:assouad}]
  Fix $\eps>0$ and notice that, by Proposition \ref{thm:conf_assouad_minimality}, there exist $x \in X$ and a rotated comb $T \in \Tan(X,x)$ such that $\dimh(T) \ge \dima(X)-\eps$. Recall that, by \cite[Theorem 4.1.11]{MackayTyson}, combs are minimal for the conformal Hausdorff dimension and hence, $\cdimh(T) = \dimh(T)$. Since, by \cite[Proposition~2.1]{Mackay}, the conformal Assouad dimension does not increase in taking tangents, we conclude
  \begin{equation*}
    \cdima(X) \geq \cdima(T) \ge \cdimh(T) = \dimh(T) \geq \dima(X) - \eps.
  \end{equation*}
  By letting $\eps \downarrow 0$, we have thus shown that $X$ is minimal for the conformal Assouad dimension.

  Let us then prove that the Assouad dimension of $X$ equals the claimed value. Since the Hausdorff dimension of $X$ can be approximated by the dimensions of $n$-step Bernoulli measures having simple Lyapunov spectrum, it follows from Theorem~\ref{thm:assouadv2} that
  \begin{equation*}
    \dima(X) \leq 1+\sup_{\atop{x\in X}{F\in X_F}} \dimh(X \cap (F+x) \cap B(x,\delta\sin(\beta/4))).
  \end{equation*}
  Furthermore, if $\nu$ is any fully supported Bernoulli measure having simple Lyapunov spectrum, then Theorem~\ref{thm:assouadv3} implies
	$$
		\dima(X) \geq 1 + \sup_{\atop{x\in X}{F\in\spt(\mu_F)}} \dimh(X\cap(F+x)\cap B(x,\delta\sin(\beta/4))),
	$$
  where $\mu_F$ is the Furstenberg measure with respect to $\nu$. Recalling Lemma~\ref{lem:xfissupp}, we get
  $$
    \dima(X) \geq 1+\sup_{\atop{x\in X}{F\in X_F}} \dimh(X \cap (F+x) \cap B(x,\delta\sin(\beta/4))).
  $$
  By compactness of $X$, every slice $X\cap(F+x)$ can be decomposed into finitely many sets of the form $X\cap(F+x)\cap B(y_i,\delta\sin(\beta/4))$, where $y_i\in X\cap(F+x)$. Since $F+x=F+y_i$, the statement follows as the Hausdorff dimension is countably stable.
\end{proof}


\end{document}